\documentclass[12pt, reqno]{amsart}
\usepackage{amsmath,amsopn,amssymb,amsthm}
\usepackage[breaklinks=true,
colorlinks=true,linkcolor=blue,citecolor=blue,urlcolor=blue
]{hyperref}

\allowdisplaybreaks %%  "\\*" WILL DISALLOW A PARTICULAR PAGE BREAK.
\DeclareMathOperator{\diam}{diam}
\DeclareMathOperator{\aff}{aff}

\DeclareMathOperator{\arsinh}{arsinh}
\newcommand{\eps}{\varepsilon}
\newcommand{\const}{\mathrm{const}}

\newtheorem{theorem}{Theorem}
\newtheorem{prop}{Proposition}
\newtheorem{lemma}{Lemma}

\newcommand{\makeclassicaltheorem}[1]{
\newtheorem*{classicaltheorem-#1}{Theorem #1}}

\makeclassicaltheorem A
\makeclassicaltheorem B
\makeclassicaltheorem C
\makeclassicaltheorem D
\makeclassicaltheorem E
\makeclassicaltheorem F
\makeclassicaltheorem G
\newtheorem*{classicaltheorem-F'}{Theorem F$'$}
\newtheorem*{classicaltheorem-G'}{Theorem G$'$}

\theoremstyle{definition}
\newtheorem{definition}{Definition}

\theoremstyle{remark}
\newtheorem{example}{Example}
\newtheorem{remark}{Remark}

\begin{document}

\title
[The infimum of the volumes of convex polytopes is 0]
{The infimum of the volumes of convex polytopes
of any given facet areas is 0}
\author
[N. V.~Abrosimov, E.~Makai, Jr., A. D.~Mednykh, Yu. G.~Nikonorov, G.~Rote]
{N. V.~Abrosimov$^*$, E.~Makai, Jr.$^{**}$, A. D.~Mednykh$^*$, \\ 
Yu. G.~Nikonorov$^{*}$, G.~Rote}

\address{N. V.~Abrosimov,
Sobolev Institute of Mathematics,
Novosibirsk, Acad.\ Koptyug Av.\ 4, 630090, RUSSIA;\newline
Novosibirsk State University,
Novosibirsk, Pirogov St.\ 2, 630090, RUSSIA}
\email{abrosimov@math.nsc.ru}

\address{E.~Makai, Jr.,
A. R\'enyi Institute of Mathematics,
Hungarian Academy of Sciences,
H-1364 Budapest, Pf.\ 127, HUNGARY}
\email{makai.endre@renyi.mta.hu, http://www.renyi.mta.hu/\~{}makai}

\address{A. D.~Mednykh,
Sobolev Institute of Mathematics,
Novosibirsk, Acad. Koptyug Av.\ 4, 630090, RUSSIA;\newline
Novosibirsk State University,
Novosibirsk, Pirogov St.\ 2, 630090, RUSSIA}
\email{mednykh@math.nsc.ru}

\address{Yu. G.~Nikonorov,
South Mathematical Institute
of the V. S. C.
of the Russian Academy of Sciences,
Vladikavkaz, Markus Str.\ 22, 362027, RUSSIA}
\email{nikonorov2006@mail.ru}

\address{G.~Rote,
Institut f\"{u}r Informatik,
Freie Universit\"at Berlin,
Takustr.~9, 14159 Berlin, GERMANY}
\email{rote@inf.fu-berlin.de}

\thanks{$^*$
  The project was supported in part by the State Maintenance Program
  for the Leading Scientific Schools of the Russian Federation (Grant
  NSH-921.2012.1) and by the Federal Target Grant ``Scientific and
  educational personnel of innovative Russia'' for 2009--2013
  (agreement no.~8206, application no.\ 2012-1.1-12-000-1003-014).
\newline
$^{**}$ Research (partially) supported by Hungarian National Foundation for 
Scientific
Research, grant nos. K68398, K75016, K81146.}

%%%%%%%%%%%%%%%%%%%%%%%%%%%%%%%%%%%%%%%%%%%%%%%%%%%%%%%%%%%%%%%%%%%%%%%%%%%

\begin{abstract}
We prove the theorem mentioned in the title for ${\mathbb{R}}^n$ where
$n \ge 3$. The case of the simplex was known previously. Also the case
$n=2$ was settled, but there the infimum was some well-defined function
of the side lengths. We also consider the cases of spherical and
hyperbolic $n$-spaces. There we give some necessary
conditions for the existence of a convex polytope with given facet areas
and some partial results about
sufficient conditions for the existence of (convex) tetrahedra.

\vspace{2mm}
(secondary).

\vspace{2mm} \noindent Keywords and phrases: convex polytopes, volume,
Euclidean, spherical, hyperbolic spaces.
\end{abstract}

\maketitle

%%%%%%%%%%%%%%%%%%%%%%%%%%%%%%%%%%%%%%%%%%%%%%%%%%%%%%%%%%%%%%%%%%%%%%%%%%%%%%
%%%%%%%%%%%%%%%%%%%%%%%%%%%%%%%%%%%%%%%%%%%%%%%%%%%%%%%%%%%%%%%%%%%%%%%%%%%%%%
%%%%%%%%%%%%%%%%%%%%%%%%%%%%%%%%%%%%%%%%%%%%%%%%%%%%%%%%%%%%%%%%%%%%%%%%%%%%%%

Studia Scientiarum Mathematicarum Hungarica

DOI: 10.1556/SSc.Math.2014.1292

%%%%%%%%%%%%%%%%%%%%%%%%%%%%%%%%%%%%%%%%%%%%%%%%%%%%%%%%%%%%%%%%%%%%%%%%%%%%%%

\section{Preliminaries}

%%%%%%%%%%%%%%%%%%%%%%%%%%%%%%%%%%%%%%%%%%%%%%%%%%%%%%%%%%%%%%%%%%%%%%%%%%%%%%

Minimum-area convex polygons with given side lengths are characterized by the 
following theorem of B\"or\"oczky--Kert\'esz--Makai, Jr.

%%%%%%%%%%%%%%%%%%%%%%%%%%%%%%%%%%%%%%%%%%%%%%%%%%%%%%%%%%%%%%%%%%%%%%%%%%%%%%

\begin{classicaltheorem-A}[\cite{BKM}]
Let $m \ge 3$ and
$s_m \ge s_{m-1} \ge \dots \ge s_1>0$ and $s_m < s_{m-1}+ \dots +s_1$.
Then the infimum of the areas of convex
$m$-gons in $\mathbb{R}^2$ that have side lengths $s_i$ equals the following
number $A$. This number $A$ is the 
minimum area of all triangles with side lengths
$\sum_{i \in I_1}s_i$, $\sum_{i \in I_2}s_i$, $\sum_{i
\in I_3}s_i$.
The
minimum is taken over
all partitions
$\{ I_1,I_2,I_3\} $ of $\{ 1,\ldots ,m \}$ into
non-empty parts
for which the three resulting side lengths
satisfy the non-strict triangle inequality.
If the cyclic order of the sides is fixed then an analogous statement holds,
where the sides with indices in each of the sets $I_1,I_2,I_3$
form
an arc of the polygonal curve.
\end{classicaltheorem-A}

%%%%%%%%%%%%%%%%%%%%%%%%%%%%%%%%%%%%%%%%%%%%%%%%%%%%%%%%%%%%%%%%%%%%%%%%%%%%%%

When we investigate simple polygons instead of  convex polygons,
we have the following result, due to B\"or\"oczky--Kert\'esz--Makai, Jr. and
Nikonorova.

%%%%%%%%%%%%%%%%%%%%%%%%%%%%%%%%%%%%%%%%%%%%%%%%%%%%%%%%%%%%%%%%%%%%%%%%%%%%%%

\begin{classicaltheorem-B}[\cite{BKM,Nikonorova}]
Let $m \ge 3$ and $s_m
\ge s_{m-1} \ge \dots \ge s_1>0$ and $s_m < s_{m-1}+ \dots +s_1$.
Then the infimum of the areas of simple
$m$-gons in $\mathbb{R}^2$ that have side lengths $s_i$ equals the following
number $B$. This number $B$ is the
minimum area of all triangles with side lengths
$\sum_{i \in I_1}\varepsilon _is_i$, $\sum_{i \in I_2}
\varepsilon _is_i$,
$\sum_{i \in I_3}\varepsilon _is_i$. 
The
minimum is taken over
all partitions
$\{ I_1,I_2,I_3\} $ of $\{ 1,\ldots ,m \}$ into
non-empty parts, and all signs $\varepsilon _1, \ldots , \varepsilon _m$,
for which the three resulting side lengths are non-negative and
satisfy the non-strict triangle inequality.

Moreover, if this minimum is not $0$ then we may additionally
suppose the following. For each $j \in \{
1,2,3 \} $ the sum $\sum_{i \in I_j}\varepsilon _is_i$ cannot be
written as $\sum_{i \in I_j'}\varepsilon _is_i +
\sum_{i \in I_j''}\varepsilon _is_i$ where $\{ I_j',I_j'' \} $ is a
partition of $I_j$ and where these partial summands are both positive.
\end{classicaltheorem-B}

%%%%%%%%%%%%%%%%%%%%%%%%%%%%%%%%%%%%%%%%%%%%%%%%%%%%%%%%%%%%%%%%%%%%%%%%%%%%%%

We remark that the proofs in the two papers were different. Moreover,
in \cite{Nikonorova}  the result is formulated in a special case only, but all
 ingredients of the proof of the general case are present in
\cite{Nikonorova} as well.

%%%%%%%%%%%%%%%%%%%%%%%%%%%%%%%%%%%%%%%%%%%%%%%%%%%%%%%%%%%%%%%%%%%%%%%%%%%%%%

In our paper we write ${\mathbb{R}}^n$, ${\mathbb{H}}^n$ and ${\mathbb{S}}^n$
for the \emph{Euclidean, hyperbolic} and \emph{spherical
$n$-space}, respectively.
Theorems A and B extend to ${\mathbb{S}}^2$ and ${\mathbb{H}}^2$ as follows.

%%%%%%%%%%%%%%%%%%%%%%%%%%%%%%%%%%%%%%%%%%%%%%%%%%%%%%%%%%%%%%%%%%%%%%%%%%%%%%

\begin{classicaltheorem-C}[\cite{BKM}]
Let $m \ge 3$ and $s_m \ge s_{m-1}
\ge \dots \ge s_1>0$ and $s_m < s_{m-1}+ \dots +s_1$.
Rather than ${\mathbb{R}}^2$ we consider ${\mathbb{H}}^2$ and ${\mathbb{S}}^2$,
but in
case of ${\mathbb{S}}^2$
we additionally suppose $\sum\limits_{i=1}^m s_i \le \pi $.
Then in both cases the word-for-word analogues of Theorems A and B hold
for ${\mathbb{H}}^2$ and ${\mathbb{S}}^2$.
\end{classicaltheorem-C}

%%%%%%%%%%%%%%%%%%%%%%%%%%%%%%%%%%%%%%%%%%%%%%%%%%%%%%%%%%%%%%%%%%%%%%%%%%%%%%

In each of these three theorems
the question of finding the infimum is
reduced to finding the minimum of a set of non-negative numbers
whose
cardinality is bounded by a function of $m$.
In Theorems A and B, this bound is 
$3^m$ and $6^m$, respectively.
In Theorem A with given
cyclic order of the sides, the bound is $\binom{m}{3}$. In Theorem C, the
 bounds are the same as for
Theorems A and~B.

%%%%%%%%%%%%%%%%%%%%%%%%%%%%%%%%%%%%%%%%%%%%%%%%%%%%%%%%%%%%%%%%%%%%%%%%%%%%%%

\smallskip
B\"or\"oczky et al.~\cite{BKM}
posed the question whether it is
possible to extend these theorems to
dimensions $n \ge 3$. Their conjecture was that, analogously to the
two-dimensional case, the solutions would be given as the volumes of some
simplices. Unfortunately, they were unaware of the fact that the case of
simplices already had long ago been solved, namely in 1938, as we will
describe below.

%%%%%%%%%%%%%%%%%%%%%%%%%%%%%%%%%%%%%%%%%%%%%%%%%%%%%%%%%%%%%%%%%%%%%%%%%%%

The analogous problem about the
maximal volume of simplices with given facet areas
(an isoperimetric-type problem)
was solved by Lagrange \cite{La}
in 1773 for ${\mathbb{R}}^3$ and 
by Borchardt \cite{Borch} in 1866 for ${\mathbb{R}}^n$. 
A simplex is called
{\emph{orthocentric}} if it has an orthocentre, i.e., a common
point of all altitudes. 
It can be characterized also as a simplex where
any two disjoint edges (or, equivalently, any two
disjoint faces of dimension at least 1)
are orthogonal. For this reason, an orthocentric simplex is sometimes also
called {\it{orthogonal}}, although orthocentric is the presently used
terminology. For a relatively 
recent exposition of the above-mentioned facts and some other properties of
orthocentric simplices, see for example
\cite{Ge}. See also 
the recent paper \cite{EHM}, whose first part is a comprehensive survey 
about orthocentric simplices. There 
 it is also stressed that for
many elementary geometrical theorems the objects in ${\mathbb{R}}^n$ 
corresponding to triangles
are not the general simplices but just the orthocentric ones.

%%%%%%%%%%%%%%%%%%%%%%%%%%%%%%%%%%%%%%%%%%%%%%%%%%%%%%%%%%%%%%%%%%%%%%%%%%%

\begin{classicaltheorem-D}
[\cite{La, Borch}]
Let $n \ge 3$. Then among the simplices in ${\mathbb{R}}^n$
with given facet areas $S_{n+1} \ge \dots  \ge S_1 >0$ 
\textup(if such simplices exist\textup)
there exists \textup(up to congruence\textup) exactly one simplex
of maximal volume. It is also the \textup(up to congruence\textup) unique
orthogonal simplex with these facet areas.
\end{classicaltheorem-D}

%%%%%%%%%%%%%%%%%%%%%%%%%%%%%%%%%%%%%%%%%%%%%%%%%%%%%%%%%%%%%%%%%%%

Unaware of the above-mentioned solution of the maximum problem,
 A.~Narasinga Rao~\cite{Rao} posed the following problem 
 in 1937:
\begin{quote}
``The areas of the four facets of a tetrahedron are $\alpha, \beta, \gamma,
\delta$. Is the volume determinate?
If not, between what limits does it lie?''
\end{quote}

This problem was soon solved independently by Venkatachaliengar, Iyengar,
Au\-luck, and Iyengar--Iyengar 
\cite{Venk, Iye, Auluck,IyeIye}.
In fact, under the above hypothesis the volume is not determined
(supposing that such
tetrahedra exist).
Moreover, they reproved that there is up to congruence
exactly one tetrahedron of maximal volume with the given facet
areas, which is also orthogonal --- and they reproved that 
it is also the unique orthocentric tetrahedron with
these facet areas. Moreover, they proved that there exists a tetrahedron
with the given face
areas that has an arbitrarily small volume.
A generalization of  
their first 
mentioned result
to multi-dimensional Euclidean spaces was
obtained in
\cite{Venk, Iye, IyeIye}. We cite only their statement about the infimum of the
volumes.

\begin{classicaltheorem-E}
[\cite{Venk, Iye, Auluck, IyeIye}]
\label{thm-E}
Let $n \ge 3$ and $S_{n+1}\geq S_n \geq \dots \geq S_1>0$.
Then there exists a simplex in ${\mathbb{R}}^n$
with these $(n-1)$-volumes of the facets
if and only if
$$
S_{n+1}< S_1 +S_2+\dots +S_n.
$$
If this inequality holds,
then, for
any $\varepsilon >0$, there is a nondegenerate simplex in ${\mathbb{R}}^n$
with facet areas $S_1,S_2,\dots,S_{n+1}$ and
 volume at most~$\eps$. 
\end{classicaltheorem-E}

%%%%%%%%%%%%%%%%%%%%%%%%%%%%%%%%%%%%%%%%%%%%%%%%%%%%%%%%%%%%%%%%%%%%%%

The proof of Theorem E by
Iyengar and Iyengar~\cite{IyeIye} was based on the
following statement, which is valid for simplices only
\cite[p.~306]{IyeIye}.
Let 
$T$ be a simplex
with facet areas $S_1,S_2,\dots,S_{n+1}$
and respective outer unit facet normals
$u_1, \dots ,u_{n+1}$. Then we have $\sum _{i=1}^{n+1} S_iu_i=0$. 
Let us consider an 
$(n+1)$-gon $P \subset
\mathbb{R}^n$ with side vectors
$S_1u_1, \dots ,S_{n+1}u_{n+1}$. Its convex hull  $T'$ is then
also a simplex, whose volume is invariant under
permutations of the side vectors of $P$. Moreover, for the volumes of
$T $ and $T '$, we have
$V(T )^{n-1}=V(T ')[(n-1)!]^2/n^{n-2}$.
Based on this relation and  some calculations,
Iyengar and Iyengar
 could make
$V(T ')$
arbitrarily small. However,
this can also be done by choosing $u_1, \dots ,u_{n+1}$ in a small
neighbourhood of the $x_1 \ldots x_{n-1}$-coordinate hyperplane. See also the
first and third proofs of our Theorem~\ref{main2}.

%%%%%%%%%%%%%%%%%%%%%%%%%%%%%%%%%%%%%%%%%%%%%%%%%%%%%%%%%%%%%%%%%%%%%%%%%%

The question of maximal volume of polytopes with given facet areas is much
less understood. 

For non-degenerate 
polytopes in ${\mathbb{R}}^n$ with given facet areas and given
facet outer unit normals, we have the following
result of Brunn~\cite{Br}, see also \cite[\S 10.5]{Mo}.
The maximal volume is attained for the (up to
translation) unique
\emph{convex} polytope with these given facet areas and given facet outer 
unit normals.
(For coinciding facet outer unit normals, one has to add their areas.)
 This result was 
rediscovered in \cite[Theorems 2 and~3]{BBMP} 
 and applied to solve another problem.
In crystallography, this set of maximal volume 
is called the {\it{Wulff shape}} \cite{Wu}. It
minimizes total surface energy of the crystal and is always convex.
For a nice description of the interplay of mathematics and crystallography see
\cite[\S 10.11]{Boro}.

For any given number $m$ of facets and fixed total surface area,
the polytope of largest volume
has an inball and
the facets must touch the inball
at their centroids (Lindel\"of's theorem, see \cite[p.~43]{St} or 
\cite[II.4.3, p.~264; English ed. IX.43, p.~283]{FT2}.

Now let us restrict 
our attention to ${\mathbb{R}}^3$. L. Fejes T\'oth
\cite[Theorem 1, p.~175]{FT1} (see also
\cite[II.4.3, p.~265, Satz; English ed. IX.43, p.~283, Theorem]{FT2})
asserts that among (convex) polyhedra with given surface area and
$m=4$, $6$, and
$12$ faces, the largest volume is attained for the regular tetrahedron,
cube, and regular dodecahedron, respectively. He gave a bound on the
maximum volume \cite[p.~175]{FT1}, valid for each
$m \ge 4$, which is also asymptotically
sharp for $m \to \infty $.
 For $m=5$, the extremal
polyhedron is the regular triangular prism that has an inball 
\cite[p.~41]{St}. A recent complete and simple proof of this fact is given
in \cite[Theorem 5.10]{GMW}.
However, for $m=8$ and $m=20$,
the extremal polyhedron is not the
regular octahedron and icosahedron, respectively \cite[p.~234]{Go}.
For more information about
this isoperimetric problem about convex polyhedra in ${\mathbb{R}}^3$
with given number of faces,
see the old survey in \cite{Go} or
the recent survey in the introduction of~\cite{TGL}. For recent numerical
results (examples) about the isoperimetric problem for polyhedra, with large
symmetry groups, see \cite{LGT}.

A different problem is to maximize the volume
enclosed by a given surface that may be bent (isometrically)
but not stretched.
A theorem of S. P. Olovianishnikoff says the following.
For convex bodies $P,Q \subset {\mathbb{R}}^3$, where $P$ is a convex 
polyhedron, any mapping of $\partial P$ to $\partial Q$ preserving the
geodesic distance of every pair of points of $\partial P$ (i.e., the length of
the shortest arc in $\partial P$ joining these points) extends
to an isometry of ${\mathbb{R}}^3$. See 
\cite[Ch. 3, \S 3,
{\bf{2}}, p.~150, Satz 1]{Alex2} for a special case, and
S. P. Olovianishnikoff
\cite{Ol}, p. 441, Theorem for the general case described above. 
For convex bodies $P,Q \subset {\mathbb{R}}^3$ where $\partial P$ is of
class $C^2$, the analogous theorem holds. See \cite[Ch. 8, \S 5,
p.~337]{Alex1} for a special case and A. V. Pogorelov
\cite[Introduction, \S 1, A, p.~8,
Theorem 1, and Ch. 3, {\bf{3}}, p.~66, Theorem 1]{Po1}
for the general case described above.

However, this uniqueness theorem does not say that this unique
convex polyhedron would have the maximal volume.
The opposite is true: the surface of \emph{every} (not necessarily convex)
polyhedron
can be isometrically deformed to increase the enclosed volume \cite{Pak}.
 For example, the cube can be
 ``blown up'': the face centers move
outwards
and the vertices move closer to the center. The face diagonals
maintain their original length, but
the original edges of the cube are longer than necessary:
they become crumpled, with wrinkles perpendicular to the original edge.
Globally, the polyhedron becomes more ``ball-like''.
 This volume-increasing
phenomenon for convex bodies was first observed by
A. V. Pogorelov in the theory of thin shells in mechanics
\cite{Po2,Po3}. (A short summary
of the results of \cite{Po2} and of some other related results is
given in
\cite{PoBa}.)
An animation showing a deformation of the cube with a volume increase by a
factor of
about 1.2567 has been produced by Buchin and Schulz~\cite{BS}.
The problem of enclosing the largest volume with the surface
of a given convex polyhedron, possibly under the constraint of
preserving the original symmetries, has been treated in many papers
\cite{THL1,THL2,Milka,BuZa,Bleecker,VAAleks,Pak,MG} (``inextensional'' in the
title of \cite{THL1} means ``isometric w.r.t. the geodesic distance'').
(According to a private communication
from the second author of \cite{MG}, 
in the tableau summarizing the numerical results in pp.~154 and~181, the values
in the middle column for the dodecahedron and the icosahedron 
are not correct. They are actually smaller than the values 
in the third column, which are proved in \cite{MG}, and
those are the best published values.)
For a recent survey on this and related questions see~\cite{Sab}.

%%%%%%%%%%%%%%%%%%%%%%%%%%%%%%%%%%%%%%%%%%%%%%%%%%%%%%%%%%%%%%%%%%%%%%%

\subsection*{Notations.}
 In this paper,
$V(\cdot )$ denotes volume of a set,
$S( \cdot )$ its surface area, $\diam( \cdot )$ its diameter,
$\aff(\cdot )$ its affine hull, 
lin\,$(\cdot )$ its linear hull, 
and
$\partial ( \cdot )$ its boundary. If we
want to indicate also
the dimension $n$ then we will write $V_n( \cdot )$ for the
$n$-volume. Sometimes we will refer to the $(n-1)$-volume in ${\mathbb{R}}^n$,
${\mathbb{H}}^n$ or ${\mathbb{S}}^n$ as \emph{area}.
We write $\kappa _n$ for the volume of the unit ball in ${\mathbb{R}}^n$.
For $x,y$ in ${\mathbb{R}}^n$, ${\mathbb{H}}^n$ or ${\mathbb{S}}^n$,
we write $[x,y]$ for the segment and
$\ell (x,y)$ for the line joining $x$ and $y$.
 On
${\mathbb{S}}^n$, $x$ and $y$ must not be antipodes, and
we mean by $[x,y]$ the minor arc on the great circle through $x$ and $y$.
The line $\ell (x,y)$ is well-defined only for $x\ne y$ --- writing 
$\ell (x,y)$ we suppose $x\ne y$.
We denote the distance between $x$ and $y$ by~$|xy|$.

%%%%%%%%%%%%%%%%%%%%%%%%%%%%%%%%%%%%%%%%%%%%%%%%%%%%%%%%%%%%%%%%%%%%%%%

For standard facts about convex bodies we refer to \cite{Schn}.

%%%%%%%%%%%%%%%%%%%%%%%%%%%%%%%%%%%%%%%%%%%%%%%%%%%%%%%%%%%%%%%%%%%%%%%%%%%
%%%%%%%%%%%%%%%%%%%%%%%%%%%%%%%%%%%%%%%%%%%%%%%%%%%%%%%%%%%%%%%%%%%%%%%%%%%%%%
%%%%%%%%%%%%%%%%%%%%%%%%%%%%%%%%%%%%%%%%%%%%%%%%%%%%%%%%%%%%%%%%%%%%%%%%%%%%%%

\section{New Results}

\subsection{Euclidean Space}

The following theorem can be considered as folklore, but we could
not locate a proof. For completeness, we state and prove it.

\begin{theorem}\label{main3}
Assume that $m>n\geq 3$ are integers,
and consider any sequence of numbers $S_{m}\geq S_{m-1} \geq
\dots \geq S_1>0$.
Then the following statements are equivalent:
\begin{enumerate}
\item [\rm(i)]
There exists a non-degenerate polytope $P\subset \mathbb{R}^n$ with $m$ facets
and with facet areas $S_1,S_2,\dots, S_m$.

\item [\rm(ii)]
There exists a non-degenerate convex polytope $P\subset \mathbb{R}^n$ with $m$
facets and
with facet areas $S_1,S_2,\dots, S_m$.

\item [\rm(iii)]
 The inequality  $S_{m}< S_1 +S_2+\dots +S_{m-1}$ holds.
\end{enumerate}
If we also allow degenerate polytopes in 
\thetag i or \thetag {ii}, 
then they imply,
rather than~\thetag {iii},
\begin{enumerate}
\item [\rm(iii$'$)]
 $S_m \le S_1+ \dots +S_{m-1}$ with equality if and only if the polytope
degenerates into the doubly counted facet with area $S_m$.
\end{enumerate}
\end{theorem}

%%%%%%%%%%%%%%%%%%%%%%%%%%%%%%%%%%%%%%%%%%%%%%%%%%%%%%%%%%%%%%%%%%%%%%%%%%%%%%

\begin{theorem}\label{main2}
Let $m>n \ge 3$ be integers.
Let $\varepsilon>0$ and 
$S_{m}\geq S_{m-1} \geq \dots \geq S_1>0$ be a sequence of numbers
such that  $S_{m}<
S_1 +S_2+\dots +S_{m-1}$. Then
there exists a non-degenerate convex
polytope $P \subset \mathbb{R}^n$ with $m$ facets and
with facet areas $S_1,S_2,\dots, S_m$
and with volume $V(P) \le \varepsilon $. 
\end{theorem}

%%%%%%%%%%%%%%%%%%%%%%%%%%%%%%%%%%%%%%%%%%%%%%%%%%%%%%%%%%%%%%%%%%%%%%%%%%%%%%

\begin{remark}\label{th2}
This theorem shows that for dimension $n \ge 3$ there are no separate
questions for convex and general polytopes. Recall that 
for dimension $n=2$ these questions had different answers, see 
Theorems {A} and {B}.
\end{remark}

%%%%%%%%%%%%%%%%%%%%%%%%%%%%%%%%%%%%%%%%%%%%%%%%%%%%%%%%%%%%%%%%%%%%%%%%%%%%%%

We give three different proofs of Theorem~\ref{main2}. The first one
is independent of Theorem~D and reproves the case of the simplex. It
is an existence proof by contradiction. The second proof uses
Theorem~{D}. It reduces the question to the case of simplices.  Both
proofs rely on delicate convergence arguments (see
Sections~\ref{sec:tools} and~\ref{Euclidean}). The third proof is geometric.  
It constructs examples with small volumes that are like ``needles''.  In
particular we will give an explicit upper bound for the volumes of
our examples in terms of the ``steepness'' of their facets
(Lemmas~\ref{volume} and~\ref{volume1}). If we consider $n,m$ and the facet
areas as fixed then
our estimate is sharp up to a constant factor (see
Lemma~\ref{volume1}).

%%%%%%%%%%%%%%%%%%%%%%%%%%%%%%%%%%%%%%%%%%%%%%%%%%%%%%%%%%%%%%%%%%%%%%%%%%%%%%

Note that there is a very interesting dichotomy. In Theorems A and B for
${\mathbb{R}}^2$ (and also in Theorem C for ${\mathbb{H}}^2$ and
${\mathbb{S}}^2$) we have some definite functions of the side lengths as
infima. In Theorem~\ref{main2} for ${\mathbb{R}}^n$ with $n \ge 3$
the infimum does not depend
at all on the facet areas.

%%%%%%%%%%%%%%%%%%%%%%%%%%%%%%%%%%%%%%%%%%%%%%%%%%%%%%%%%%%%%%%%%%%%%%%%%%%%%%

\subsection{Hyperbolic Space}

For the hyperbolic case we have a word-for-word analog of the implications 
(ii)$\Longrightarrow $(iii) and (ii)$\Longrightarrow $(iii$'$) from 
Theorem~\ref{main3}
(under the respective hypotheses). 

\begin{prop}\label{necne}
Let $P \subset {\mathbb{H}}^n$ be a polytope with facet areas
$S_m \ge S_{m-1} \ge \dots  \ge S_1 > 0$. Then the inequality
$S_m \le \sum\limits_{i=1}^{m-1} S_i$ holds,
with equality if and only if
$P$ degenerates into the doubly counted facet with area $S_m$.
\end{prop}

%%%%%%%%%%%%%%%%%%%%%%%%%%%%%%%%%%%%%%%%%%%%%%%%%%%%%%%%%%%%%%%%%%%%%%%%%%%%%%

Next we give two statements that show the following. The necessary condition in
Proposition~\ref{necne} together with the inequalities $S_i \le \pi $
is not sufficient even for the existence of a tetrahedron
in ${\mathbb{H}}^3$ with these facet areas.
That is, there are some further necessary conditions.
Recall that the area of a simple $k$-gon in
${\mathbb H}^2$ is bounded by $(k-2)\pi $.

%%%%%%%%%%%%%%%%%%%%%%%%%%%%%%%%%%%%%%%%%%%%%%%%%%%%%%%%%%%%%%%%%%%%%%%%%%%%%%

\begin{prop}\label{necne'}
Let us admit polyhedra in
${\mathbb H}^3$ whose vertices are all distinct but which possibly
have some
infinite vertices. Then a non-degenerate
polyhedron with facet areas $S_m,S_{m-1}, 
\dots , S_3$ maximal \textup(i.e., $(k-2)\pi $ for a $k$-gonal face\textup) 
but with
facet areas $S_2,S_1$ not maximal
does not exist.
\end{prop}

%%%%%%%%%%%%%%%%%%%%%%%%%%%%%%%%%%%%%%%%%%%%%%%%%%%%%%%%%%%%%%%%%%%%%%%%%%%%%%

Proposition \ref{necne'} 
would suggest that for polyhedra in ${\mathbb H}^3$, if all facets but
two have areas nearly maximal (i.e., close to
$(k-2) \pi $ for a $k$-gonal face)
then the same statement would hold for the remaining two facets as well.
However, this is not true. Even in the convex case,
these two facets can have areas close to $0$, as shown by the following example.
Consider a very large circle in $\mathbb{H}^2 \subset {\mathbb H}^3$
and a regular $l$-gon
$p_1 \ldots p_l$ inscribed in it ($l \ge 3$).
 Choose $p_{l+1}$ on our circle with
$|p_lp_{l+1}|=\varepsilon $. Then all triangles with vertices among the
$p_i$'s have areas close to $\pi $ except those that contain both $p_l$ and
$p_{l+1}$, and 
those have very small areas. Now perturb these points $p_i$ a little
bit in ${\mathbb H}^3$ so that no four 
lie in a plane. Then
their convex hull is a triangle-faced convex
polyhedron, and the perturbation of the
segment $[p_l,p_{l+1}]$ is an edge of it. (To see this, use the 
collinear model. For any
convex polygon with strictly convex angles,
its edges will remain edges of the convex hull after a sufficiently small 
perturbation.)
The two facets of our polyhedron
incident to this edge have very small areas while all other facets have areas
close to $\pi $, i.e., are nearly maximal.

%%%%%%%%%%%%%%%%%%%%%%%%%%%%%%%%%%%%%%%%%%%%%%%%%%%%%%%%%%%%%%%%%%%%%%%%%%%%%%

However, an analogous statement for all but one facets will be shown in the
convex case.

%%%%%%%%%%%%%%%%%%%%%%%%%%%%%%%%%%%%%%%%%%%%%%%%%%%%%%%%%%%%%%%%%%%%%%%%%%%%%%

\begin{prop}\label{nonexhyp}
Assume that we have a convex
polyhedron in $\mathbb{H}^3$ with infinite vertices
admitted. Suppose its $m$
facets are  a
$k_m$-gon, $\ldots $, $k_1$-gon and have
respective areas
$S_m \ge \dots  \ge S_2 \ge S_1 >0$.
Then for any $i \in \{ 1, \dots ,m
\} $ we have
$$
(k_i -2)\pi - S_i \le \sum 
_{\substack {1\le j\le m\\ j\ne i}}
(  (k_j-2) \pi - S_j ).
$$
If there is a finite vertex whose incident edges do not lie in a plane, then the
above inequality is strict.
\end{prop}

%%%%%%%%%%%%%%%%%%%%%%%%%%%%%%%%%%%%%%%%%%%%%%%%%%%%%%%%%%%%%%%%%%%%%%%%%%%%%%

In \S 6 Remark 9,
it will be explained that, in a sense, there are no
interesting analogues of Proposition~\ref{nonexhyp}
for $\mathbb{R}^3$ and $\mathbb{S}^3$.

%%%%%%%%%%%%%%%%%%%%%%%%%%%%%%%%%%%%%%%%%%%%%%%%%%%%%%%%%%%%%%%%%%%%%%%%%%%%%%

Now we turn to sufficient conditions for the existence of hyperbolic
tetrahedra.

%%%%%%%%%%%%%%%%%%%%%%%%%%%%%%%%%%%%%%%%%%%%%%%%%%%%%%%%%%%%%%%%%%%%%%%%%%%%%%

\begin{theorem}\label{exhyp}
Assume that $\pi/2>S_4\geq S_3\geq S_2\geq S_1 >0$,
 $S_4 <S_1+S_2+S_3$, and
one of the inequalities
\begin{equation}\label{fc1}
\tan(S_1/2)> \frac{1-\cos S_4}{2\sqrt{\cos S_4}},
\end{equation}
or
\begin{equation}\label{fc2}
S_4\geq S_3+S_2
\end{equation}
holds.
Then there exists a non-degenerate
tetrahedron $T \subset {\mathbb{H}}^3$ with facet areas
$S_1,\allowbreak S_2,\allowbreak S_3,\allowbreak S_4$. 
\end{theorem}

%%%%%%%%%%%%%%%%%%%%%%%%%%%%%%%%%%%%%%%%%%%%%%%%%%%%%%%%%%%%%%%%%%%%%%%%%%%%%%

\subsection{Spherical Space}
\label{spherical}
For the spherical case, we give some necessary and some sufficient 
conditions for the existence.

We say that a set
$X \subset {\mathbb{S}}^n$ (for $n \ge 2$) is \emph{convex} if, for
any two non-antipodal $x,y \in X$, the connecting minor great-${\mathbb{S}}^1$
arc $[x,y]$
also belongs to $X$.
This definition classifies an
antipodal pair of points as a convex set.
But 
these are the only convex sets which are disconnected, 
and since the sets we consider 
contain non-trivial arcs,
these exceptional cases play no role for us.
By a \emph{nondegenerate simplex in ${\mathbb{S}}^n$} we mean the set of 
those points of
${\mathbb{S}}^n$ that have non-negative coordinates in some (non-orthogonal)
coordinate
system with origin at $0$, with its usual face lattice.
A \emph{simplex in ${\mathbb{S}}^n$} is a nondegenerate simplex, or
a limiting
position of nondegenerate simplices. Thus, for example, we will not consider
concave spherical triangles or spherical triangles
with sides $3 \pi /2, \pi /4, \pi /4$ or $2 \pi , 0, 0$,
but a spherical triangle with angles $\pi,\pi/5,\pi/5$ and sides
$\pi,\pi/3,2\pi/3$ is a (degenerate) simplex. As a point set, this
simplex
is indistinguishable from a digon. A different division of the digon side,
like $\pi,\pi/4,3\pi/4$, is regarded as a different simplex.
To emphasize the fact that we do not just regard a simplex as a point
set but we consider its face structure, we will often refer to it as a
\emph{combinatorial simplex}.
All simplices in $\mathbb{S}^n$, as well as
in ${\mathbb{R}}^n$ and ${\mathbb{H}}^n$, are convex.
A simplex in an open half-${\mathbb{S}}^n$ with non-empty interior
is always nondegenerate.
(Observe that an open
half-${\mathbb{S}}^n$ also has a collinear model in ${\mathbb{R}}^n$ that also
respects convexity. For the open southern half-${\mathbb{S}}^n$ in
${\mathbb{R}}^{n+1}$ consider the central projection to the tangent space
${\mathbb{R}}^n$
at the South Pole.)

%%%%%%%%%%%%%%%%%%%%%%%%%%%%%%%%%%%%%%%%%%%%%%%%%%%%%%%%%%%%%%%%%%%%%%%%%%%%%%

\begin{prop}\label{sphnec}
Let $P \subset {\mathbb{S}}^n$ be a polytope with facet
areas $S_m \ge \dots \ge S_1>0$, such that each facet lies in some closed
half-$\,{\mathbb{S}}^{n-1}$. Then
$$
S_m \le S_1+ \dots +S_{m-1}.
$$
Here strict inequality holds if $P$ is contained in an open
half-$\,\mathbb{S}^n$ and does not degenerate into the 
doubly-counted facet with area $S_m$.

If $P$ is a convex polytope contained in some
closed half-${\mathbb{S}}^n$, then
$$
\sum _{i=1}^m S_i \le V_{n-1}({\mathbb{S}}^{n-1}) .
$$
Here strict inequality holds if $P$ is contained in an open
half-$\,\mathbb{S}^n$.
\end{prop}

%%%%%%%%%%%%%%%%%%%%%%%%%%%%%%%%%%%%%%%%%%%%%%%%%%%%%%%%%%%%%%%%%%%%%%%%%%%%%%

\begin{remark}\label{mark7}
Clearly, in the first part of Proposition \ref{sphnec},
the hypothesis that each facet lies in some closed half-$\mathbb{S}^n$
cannot be
dispensed.  Already for $n=2$, we may  even have a degenerate combinatorial 
simplex lying
in some great-$\mathbb{S}^{n-1}$ with one facet strictly containing a
half-$\mathbb{S}^{n-1}$.
 In the second part of Proposition \ref{sphnec}, if
$P$ is contained in a closed half-$\mathbb{S}^n$ but
 not in an open half-$\mathbb{S}^n$, then
 equality can occur:
$P$ can degenerate so that one facet is a closed
half-$\mathbb{S}^{n-1}$, and the union of the other
facets is this closed half-$\mathbb{S}^{n-1}$ or the closure of its
complement in this ${\mathbb{S}}^{n-1}$. 
\end{remark}

%%%%%%%%%%%%%%%%%%%%%%%%%%%%%%%%%%%%%%%%%%%%%%%%%%%%%%%%%%%%%%%%%%%%%%%%%%%%%%

\begin{remark}\label{remark3-new}
We do not know how to algorithmically decide whether a
simplex with given facet areas $S_i$ in
  $\mathbb{H}^n$ or $\mathbb{S}^n$ exists, for $n
\ge 3$.
The main difficulty are the transcendental functions
that enter into the calculation of volumes.
In $\mathbb{H}^3$ and $\mathbb{S}^3$,
however,
we have a positive answer to
a slightly modified question. 
The question whether
there is a tetrahedron 
(for ${\mathbb{S}}^3$
in the sense described above) with 
facet areas $S_1,S_2,S_3,S_4$,
is decidable if we are given
$\tan (S_1/2), \dots ,\tan (S_4/2)$ as inputs.

We model this question by setting up a system of equations
and inequalities in
the unknown coordinates $(x_{ij})$ of the four vertices. (For ${\mathbb{S}}^3$
we use its standard embedding into ${\mathbb{R}}^4$, while for
${\mathbb{H}}^3$ we use the hyperboloid model in ${\mathbb{R}}^4$.)
The equations
express the condition that the vertices lie on
 ${\mathbb{S}}^3$ or
 ${\mathbb{H}}^3$, and that
the facet areas of the
corresponding tetrahedron should be $S_1,S_2,S_3,S_4$.
Further inequalities are necessary for ${\mathbb{S}}^3$ to
ensure our definition of simplices.
We are interested in the set of 4-tuples
$(S_1,S_2,S_3,S_4)$ for which \emph{there exist} coordinate
vectors $(x_{ij})$ that fulfill the conditions.
These conditions turn out to be polynomial equations
and inequalities (these polynomials having rational coefficients)
in
$\tan (S_1/2), \dots ,\tan (S_4/2)$ and in the coordinates $x_{ij}$.
By a fundamental result of
Tarski~\cite{T}, this existence question is therefore
(in principle) decidable (we can eliminate the variables $x_{ij}$).
More specifically, the set of quadruples 
$(\tan (S_1/2), \dots ,\tan (S_4/2))$ for all tetrahedra in 
$\mathbb{H}^3$ or $\mathbb{S}^3$ (for $\mathbb{S}^3$ in our sense)
can be described by a finite number of
polynomial equalities and inequalities, these polynomials having rational
coefficients, also using the usual logical connectives ``and'', ``or'',
``not''.
In other words, this set forms a {\emph{semi-algebraic set}}.
\end{remark}

%%%%%%%%%%%%%%%%%%%%%%%%%%%%%%%%%%%%%%%%%%%%%%%%%%%%%%%%%%%%%%%%%%%%%%%%%%%

\begin{remark}
{For simplices in $\mathbb{S}^n$ and $\mathbb{H}^n$ (with finite vertices)
the case $S_1=0$ and all other $S_i$'s positive and sufficiently small can be
described. The description is: there exists 
a partition of the other facets into two classes such
that for the two classes the sums of the facet areas are equal.
For this we have to use index considerations, like later 
in the Proofs of Theorems 3 and 4.} 
\end{remark}

%%%%%%%%%%%%%%%%%%%%%%%%%%%%%%%%%%%%%%%%%%%%%%%%%%%%%%%%%%%%%%%%%%%%%%%%%%%%%%

\begin{theorem}\label{exsph}
Assume that $\pi/2>S_4\geq S_3\geq S_2\geq S_1 >0$,
 $S_4 <S_1+S_2+S_3$, and
one of the inequalities
\begin{equation}\label{fc1'}
\tan(S_1/2) \ge \frac{1-\cos S_4}{2\sqrt{\cos S_4}},
\end{equation}
or
\begin{equation}\label{fc2'}
S_4\geq S_3+S_2
\end{equation}
holds.
Then there exists a non-degenerate \textup(convex\textup)
tetrahedron $T \subset {\mathbb{S}}^3$ with facet areas $S_1,S_2,S_3,S_4$.
\end{theorem}

%%%%%%%%%%%%%%%%%%%%%%%%%%%%%%%%%%%%%%%%%%%%%%%%%%%%%%%%%%%%%%%%%%%%%%%%%%%%%%

Now we turn to sufficient conditions for the existence of spherical
polyhedral complexes. The second statement of Proposition \ref{sphsuff} says
the following. For combinatorial simplices contained in some closed
half-${\mathbb{S}}^n$, the two necessary conditions from Proposition
\ref{sphnec} %for their existence
 are also sufficient for their existence.

%%%%%%%%%%%%%%%%%%%%%%%%%%%%%%%%%%%%%%%%%%%%%%%%%%%%%%%%%%%%%%%%%%%%%%%%%%%%%%

\begin{prop}\label{sphsuff}
  \begin{enumerate}
  \item [\rm(i)]
Let $n \ge 2$ and $m \ge 3$ be integers and let $S_m \ge \dots \ge S_1>0$
and $S_m \le S_1+ \dots +S_{m-1}$ and $S_1+ \dots +S_m \le
V_{n-1}({\mathbb{S}}^{n-1})$. Then there exists a convex $n$-dimensional
polyhedral complex in ${\mathbb S}^n$ that lies 
in a closed half-${\mathbb S}^n$ and has
facet areas
$S_1, \dots , S_m$. All of its facets have two $(n-2)$-faces. If
$S_m < S_1+ \dots +S_{m-1}$ and $S_1+ \dots +S_m <
V_{n-1}({\mathbb{S}}^{n-1})$ then 
all its dihedral angles are less than $\pi $.

\item [\rm(ii)] Let $n \ge 2$ 
  and assume $S_{n+1} \ge \dots \ge S_1>0$, $S_{n+1} \le S_1+ \dots +
  S_n$, and $S_1+ \dots + S_{n+1} \le V_{n-1}({\mathbb{S}}^{n-1})$.
  Then there exists a convex polyhedral complex in ${\mathbb S}^n$
  lying in a closed half-${\mathbb S}^n$ that is a combinatorial
  $n$-simplex with facet areas $S_1, \dots , S_{n+1}$. Its faces of
  any dimension \textup(including the complex itself\textup) have
  their dihedral angles at most $\pi $, and are thus convex, but some
  of their dihedral angles are equal to $\pi $ for $n \ge 3$.
  \end{enumerate}
\end{prop}

%%%%%%%%%%%%%%%%%%%%%%%%%%%%%%%%%%%%%%%%%%%%%%%%%%%%%%%%%%%%%%%%%%%%%%%%%%%%%%
%%%%%%%%%%%%%%%%%%%%%%%%%%%%%%%%%%%%%%%%%%%%%%%%%%%%%%%%%%%%%%%%%%%%%%%%%%%%%%
%%%%%%%%%%%%%%%%%%%%%%%%%%%%%%%%%%%%%%%%%%%%%%%%%%%%%%%%%%%%%%%%%%%%%%%%%%%%%%

\section{Tools for the Euclidean case: Minkowski's theorems}
\label{sec:tools}

%%%%%%%%%%%%%%%%%%%%%%%%%%%%%%%%%%%%%%%%%%%%%%%%%%%%%%%%%%%%%%%%%%%%%%%%%%%%%%

We recall some classical concepts and theorems, which are in essence due to
Minkowski, but got their final form by A. D. Aleksandrov \cite{Aleks} and
W. Fenchel and B. Jessen \cite{FJ}.
 We state the results first for arbitrary convex bodies,
and then we restrict them to convex polytopes.
 We will actually need general convex bodies
when considering convergent sequences of convex polytopes
in our first two proofs of Theorem~\ref{main2}
(Sections~\ref{sec:first} and~\ref{sec:second}).
The third proof uses only 
Minkowski's Theorem about convex polytopes
(Theorem~F$'$). The reader may want to skip
directly to Theorem~F$'$.

A \emph{convex body in
${\mathbb{R}}^n$} is a compact convex set $K \subset {\mathbb{R}}^n$
with interior points. For $x \in \partial K$ we say that
$u \in {\mathbb{S}}^{n-1}$ is an
\emph{outer unit normal vector for $K$ at $x$} if $\max \{ \langle k,u
\rangle \mid k \in K \} = \langle x,u \rangle $. In this section we assume
$n \ge 2$ although the theorems of this section will be applied later
for $n \ge 3$ only.

\begin{definition}[{Minkowski, Aleksandrov \cite{Aleks},
Fenchel--Jessen \cite{FJ}, see also \cite[p.~207, (4.2.24) (with $\tau
(K, \omega )$ defined on  p.~77)]{Schn}}]\label{surfareameas}
Let $K \subset {\mathbb{R}}^n$
be a convex body. The \emph{surface area measure
$\mu _K$ of $K$} is a finite Borel measure on ${\mathbb{S}}^{n-1}$
defined as
follows. For a Borel set $B \subset {\mathbb{S}}^{n-1}$,
 $\mu _K (B)$ is the
$(n-1)$-dimensional Hausdorff measure of the set 
$\{ x \in \partial K \mid $ there
is an outer unit normal vector $u$ to $K$ at $x$ such that $u \in B \} $.
\end{definition}

Thus, $\mu _K$ is an element of $C({\mathbb{S}}^{d-1})^*$,
the dual space of the space of real-valued
continuous functions $C({\mathbb{S}}^{d-1})$ on ${\mathbb{S}}^{d-1}$,
i.e., the finite signed Borel measures on
${\mathbb{S}}^{d-1}$. We will use the weak$^*$ topology of
$C({\mathbb{S}}^{d-1})^*$
as the topology for the finite (signed) Borel measures $\mu _K$. That is,
convergence of a sequence (or more generally of a net) of finite signed Borel
measures $\mu _{\alpha } \in C({\mathbb{S}}^{d-1})^*$ to a finite signed Borel
measure $\mu \in C({\mathbb{S}}^{d-1})^*$
means the following. 
For each $f \in C({\mathbb{S}}^{n-1})$, we have 
$\int _{{\mathbb{S}}^{n-1}} f(u)d \mu _{\alpha } (u) \to \int
_{{\mathbb{S}}^{n-1}} f(u)d \mu (u)$.
Moreover, since ${\mathbb{S}}^{n-1}$ is a
compact metric space,
the space $C({\mathbb{S}}^{n-1})$ is separable, and
hence the weak$^*$ topology of
$C({\mathbb{S}}^{n-1})^*$ is metrizable. Therefore, it suffices to give the
convergent sequences in it (i.e., 
 it is
not necessary to consider nets).

For these elementary concepts and facts from functional analysis, we refer
to \cite{DS}.

%%%%%%%%%%%%%%%%%%%%%%%%%%%%%%%%%%%%%%%%%%%%%%%%%%%%%%%%%%%%%%%%%%%%%%

\begin{classicaltheorem-F}
[{Minkowski, Aleksandrov \cite{Aleks},
Fenchel--Jessen \cite{FJ}, see~also \cite[p.~389, (7.1.1), pp.~389--390,
p.~392, Theorem 7.1.2., p.~397, Theorem 7.2.1]{Schn}}]
Let
$n \ge 2$ be an integer and $K \subset {\mathbb{R}}^n$ a convex body.
The measure $\mu _K$ defined in Definition~\ref{surfareameas}
is invariant under translations of $K$
and has the following properties.
\begin{enumerate}
\item [\thetag i]
$ \int _{{\mathbb{S}}^{n-1}} u d \mu _K (u) = 0$, and
\item [\rm(ii)]
 $\mu _K$ is not concentrated on any great-${\mathbb{S}}^{n-2}$ of
${\mathbb{S}}^{n-1}$.
\end{enumerate}
Conversely, for any finite Borel measure $\mu $ on ${\mathbb{S}}^{n-1}$
satisfying \thetag i and \thetag {ii},
there exists a convex body $K$ such that $\mu _K = \mu $. Moreover, this
convex body $K$ is unique up to translations.
\end{classicaltheorem-F}

Thus,  we can consider the map $K \mapsto \mu _K$ also as a map $\{ $translates
of $K \} \mapsto \mu _K$. 

%%%%%%%%%%%%%%%%%%%%%%%%%%%%%%%%%%%%%%%%%%%%%%%%%%%%%%%%%%%%%%%%%%%%%

\begin{classicaltheorem-G}
[{Minkowski, Aleksandrov \cite{Aleks},
Fenchel--Jessen \cite{FJ}, see also \cite[p.~198, Theorem 4.1.1,
p.~205, pp.~392--393, proof of Theorem 7.1.2]{Schn}}]
Let $n \ge 2$ be an integer.
Then the mapping $\{ $translates of $K \} \mapsto \mu _K$
defined in Definition \ref{surfareameas} and just before this theorem
is a homeomorphism between its domain and its range. Its domain is
the quotient topology
of the topology on the convex bodies induced by the Hausdorff metric
with respect to the equivalence relation of being translates.
Its range is the set of finite Borel measures on ${\mathbb{S}}^{n-1}$
satisfying \thetag i and \thetag {ii} of Theorem~F with the subspace topology
of the weak\,$^*$ topology on $C({\mathbb{S}}^{n-1})^*$.
\end{classicaltheorem-G}

%%%%%%%%%%%%%%%%%%%%%%%%%%%%%%%%%%%%%%%%%%%%%%%%%%%%%%%%%%%%%%%%%%%%%

We have to remark that the cited sources, 
\cite[pp.~392--393, proof of Theorem 7.1.2]{Schn},
as well as \cite[proof of the theorem on p.~36, on p.~38]{Aleks}, 
contain explicitly only the proof of the continuity
of the bijection $\{ $translates of $K \} \mapsto \mu _K$. However, also the
continuity of the inverse map is proved at both places
although not explicitly stated.
In fact, as kindly
pointed out to the authors by R. Schneider, one has to make the
following addition to his book \cite[proof of Theorem 7.1.2]{Schn}.
Let the sequence of
surface area measures $\mu _{K_i}$
of some convex bodies $K_i \subset {\mathbb{R}}^n$
converge to the surface area measure $\mu _K$ of some convex body
$K\subset {\mathbb{R}}^n$
in the weak$^*$ topology. Then all $K_i$'s
have a bounded diameter. This is stated there
for polytopes only, but the given proof 
is valid for all convex bodies. By a
translation one can achieve that all $K_i$'s and also $K$
are contained in a fixed
ball. 
Recall that the set of non-empty compact convex sets contained
in some closed
ball is compact in their usual topology (i.e., that of the Hausdorff metric). 
Therefore,
we can choose a convergent subsequence $K_{i_j}$ of $K_i$ with limit $K'$,
say. In the Note added in proof at the end of the paper we will show that 
$K'$ is a convex body.
Then the surface area measure $\mu _{K'}$ of $K'$ is
the weak$^*$ limit of the
$\mu _{K_{i_j}}$'s, i.e., it equals the originally considered $\mu _K$.
Hence we have that
$K'$ is a translate of $K$.
Then the entire sequence $K_i$ converges to $K$. Otherwise, we
could choose another subsequence $K_{i_k}$ converging to another convex body
$K''$, which is not a translate of $K$, 
with $\mu _{K''} = \mu _K$. This is a contradiction. Then also the
translation equivalence class of $K_i$ converges to that of $K$, by
continuity of the quotient map.

%%%%%%%%%%%%%%%%%%%%%%%%%%%%%%%%%%%%%%%%%%%%%%%%%%%%%%%%%%%%%%%%%%%%%

It was also proved by Minkowski that a convex body 
$K$ is a convex polytope if and only if
$\mu _K$ (that satisfies (i) and (ii) of Theorem F) has a finite
support \cite[p.~390, Theorem 7.1.1, also considering p.~397, 
Theorem 7.2.1]{Schn}. 
If the support is $\{ u_1, \dots ,u_m \} $, we may write
$$
\mu _K = \sum _{i=1}^m \mu _K (\{ u_i \} ) \delta (u_i),
$$
where $\delta (u_i)$ is the \emph{Dirac measure concentrated at $u_i$}. (I.e.,
for a Borel set $B \subset {\mathbb{S}}^{n-1}$ we have
$\delta (u_i) (B)=0 \Longleftrightarrow u_i \not\in B$ and
$\delta (u_i) (B)=1 \Longleftrightarrow u_i \in B$.)
When we write such an equation,
we always assume that $\mu _K (\{ u_i \} ) \ne 0$ for all $i \in \{ 1, \dots
,m \} $. (Thus, the empty sum means the $0$ (finite signed Borel) measure;
although for a convex body $K$, we have $\mu _K \ne 0$.)
The weak$^*$ topology restricted to the finite signed Borel
measures of finite support,
where the support has at most $m$ elements, is the following. 
(We will use only the case when we have a
finite Borel measure and (i) and (ii) of Theorem F hold.) 
For $u_{\alpha }, u \in
{\mathbb{S}}^{d-1}$ with
$u_{\alpha } \to u$ and for $c_{\alpha },
c \in {\mathbb{R}} \setminus \{ 0 \} $
with $c_{\alpha } \to c$, where the $u _ {\alpha }$'s and $c _ {\alpha }$'s
are nets indexed by $\alpha $'s from the same index set,
we have
$c_{\alpha } \delta (u_{\alpha }) \to  c \delta (u)$. Moreover,
for arbitrary $u _{\alpha }
\in {\mathbb{S}}^{n-1}$ and $c_{\alpha } \to 0$, we have
$c_{\alpha } \delta (u_{\alpha }) \to 0$.
Thus, the convergence is defined for finite signed Borel
measures whose supports have at most one point.
Then the convergence is defined for finite sums of such sequences as well
(and in fact, only for these, see the formal definition in the next paragraph).

More exactly, a sequence (or more generally, a net)
$\mu _{\alpha }= \sum _{i=1}^{m_{\alpha }} \mu _{\alpha }
(\{ u_i \} ) \delta (u_i)$ of finite signed
Borel measures on ${\mathbb{S}}^{d-1}$
with $m_{\alpha } \le m$ can converge only to a finite signed Borel
measure of support of at most $m$ points. Moreover, $\mu _{\alpha }$ 
tends to a finite signed Borel
measure $\mu = \sum _{i=1}^{m'} \mu (\{ u_i \} ) \delta (u_i)$ on
${\mathbb{S}}^{n-1}$
with $1 \le m' \le m$ if and only if the following holds. For each $\alpha$,
there exists a
partition of $\{ 1, \dots ,m_{\alpha } \} $
of cardinality $m'$, say $\{ P_{\alpha 1}, \dots ,
P_{\alpha m'} \} $ (where each $P_{\alpha j}$ is non-empty), such that
\begin{enumerate}
\item [(A)]
 for any $j \in \{ 1, \dots ,m' \} $, the
sets $P_{\alpha j}$
converge to $u_j$ (i.e., for any neighbourhood $U_j$ of $u_j$ and for all
sufficiently large $\alpha $, we have
 $P_{\alpha j} \subset U_j$) and
\item [(B)]
 for any $j \in \{ 1, \dots ,m' \} $, the sum
$\sum \{ \mu _{\alpha }(\{ u_i \} )
\mid i \in P_{\alpha j} \} $ converges to $\mu (\{ u_j \} )$.
\end{enumerate}
The same sequence (or more generally a net) $\mu _{\alpha }$
tends to the $0$ (finite signed Borel) measure
if and only if
\begin{enumerate}
\item [(C)]
$
\sum_{i=1}^{m_ \alpha } | \mu _{\alpha }( \{ u_i \} ) | 
\to 0$. (This corresponds to the
case $m'=0$, and also here, an empty sum means $0$.)
\end{enumerate}

%%%%%%%%%%%%%%%%%%%%%%%%%%%%%%%%%%%%%%%%%%%%%%%%%%%%%%%%%%%%%%%%%%%%%%%5

For convex polytopes, Theorem F can be rewritten for
$\mu _K = \sum _{i=1}^m \mu _K (\{ u_i \} )\delta (u_i) $ as follows.

\begin{classicaltheorem-F'}
[{Minkowski, see also \cite[p.~389, (7.1.1), pp.~389--390, 
p.~390, Theorem 7.1.1, p.~397, Theorem 7.2.1]{Schn}}]
Let $m > n \ge 2$ be integers, let $S_1, \dots ,S_m >0$,
and let $u_1, \dots ,u_m \in {\mathbb{S}}^{n-1}$.
Then there exists a non-degenerate
convex polytope $P$ having $m$ facets with facet areas
$S_1, \dots ,S_m$ and respective
facet outer unit normals $u_1, \dots ,u_m$ if and only if
\begin{enumerate}
\item [\rm(i)]
$ \sum _{1=1}^m S_iu_i = 0$, and
\item [\rm(ii)]
$u_1, \dots ,u_m$
do not lie in a linear $(n-1)$-subspace of ${\mathbb{R}}^n$.
\end{enumerate}
Moreover,
if $P$ exists, it
is unique up to translations.
\end{classicaltheorem-F'}

%%%%%%%%%%%%%%%%%%%%%%%%%%%%%%%%%%%%%%%%%%%%%%%%%%%%%%%%%%%%%%%%%%%%%%%5

For convex polytopes with at most $m$ facets, Theorem G can be rewritten for
$\mu _K = \sum _{i=1}^m \mu _K (\{ u_i \} ) \delta (u_i) $ as follows.

\begin{classicaltheorem-G'}
[{Minkowski, see also \cite[p.~198, Theorem 4.1.1,
p.~205, pp.~392--393, proof of Theorem 7.1.2]{Schn} and the addition after our
Theorem G}]
Let $m > n \ge 2$ be integers.
Then the mapping $\{ \mbox{translates of }\,K \} \mapsto \mu _K$
defined in Definition~\ref{surfareameas} and after Theorem F
is a homeomorphism between its domain and its range.
Its domain is
the subspace corresponding to
the non-degenerate convex polytopes with at most $m$ facets
of the quotient topology
of the topology on the convex bodies 
\textup(induced by the Hausdorff metric\textup)
with respect to the equivalence relation of being translates. Its range is
the set of finite 
Borel measures on ${\mathbb{S}}^{n-1}$ with supports of at most $m$
points
satisfying \thetag {i} and \thetag {ii} of Theorem F\/$'$, with the subspace 
topology
of the weak\,$^*$ topology on $C({\mathbb{S}}^{d-1})^*$. This subspace
topology is described in more explicit form
 before Theorem F\/$'$.
\end{classicaltheorem-G'}

%%%%%%%%%%%%%%%%%%%%%%%%%%%%%%%%%%%%%%%%%%%%%%%%%%%%%%%%%%%%%%%%%%%%%%%%%%%%%%
%%%%%%%%%%%%%%%%%%%%%%%%%%%%%%%%%%%%%%%%%%%%%%%%%%%%%%%%%%%%%%%%%%%%%%%%%%%%%%
%%%%%%%%%%%%%%%%%%%%%%%%%%%%%%%%%%%%%%%%%%%%%%%%%%%%%%%%%%%%%%%%%%%%%%%%%%%%%%

\section{Proofs for the Euclidean case}
\label{Euclidean}

%%%%%%%%%%%%%%%%%%%%%%%%%%%%%%%%%%%%%%%%%%%%%%%%%%%%%%%%%%%%%%%%%%%%%%%%%%%%%%

Essentially the following proposition was used in \cite{IyeIye} without
explicitly stating and proving it. It can be considered as folklore (as part of
the proof of the folklore 
Theorem 1), but we state and prove it for completeness.

\begin{prop}\label{main}
Let $m>n \ge 3$ be integers.
Let $\delta>0$ and let
$S_{m}\geq S_{m-1} \geq \dots \geq S_1>0$ be numbers such that  $S_{m}<
S_1 +S_2+\dots +S_{m-1}$.
Then there are pairwise distinct unit vectors
$v_1,v_2, \dots , v_{m} \in \mathbb{S}^{n-1}$
with the following properties\textup:
\begin{enumerate}
\item 
[\rm(i)]
they lie in the open $\delta$-neighbourhood
of the $x_1x_2$-coordinate plane,
\item [\rm(ii)] they do not lie in a linear $(n-1)$-subspace 
of $\mathbb{R}^n$, and
\item [{\rm{(iii)}}] $S_1 v_1+S_2 v_2+ \dots +S_{m}v_{m}=0$.
\end{enumerate}
\end{prop}

%%%%%%%%%%%%%%%%%%%%%%%%%%%%%%%%%%%%%%%%%%%%%%%%%%%%%%%%%%%%%%%%%%%%%%%%%%%%%%

\begin{proof}
Let $P$ be the $x_1x_2$-coordinate plane in $\mathbb{R}^n$.
Since $S_m< S_1 +S_2+\dots +S_{m-1}$,
there exists a convex polygon 
$A_1A_2 \ldots A_m$ (with angles strictly smaller than $\pi $) in
$P$ such that $| A_iA_{i+1} | =S_i$ for $i=1, \dots ,m$ (indices considered
modulo $m$), see \cite[p.~44]{JB}, \cite[pp.~53--54]{Kr}.
Then the edge directions $u_i:=\overrightarrow{A_iA_{i+1}} / | A_iA_{i+1} | 
\in {\mathbb S}^{n-1} \cap P $ are distinct unit vectors.
We will perturb
 $A_1A_2 \ldots A_m$ to a spatial polygon
 $B_1B_2 \dots B_m$, keeping the side lengths equal:
$| B_iB_{i+1} | =| A_iA_{i+1} | =S_i$.
The unit vectors $v_i :=
\overrightarrow{B_iB_{i+1}} / S_i$
will then fulfill
(iii) by construction.

Clearly, for $\| v_i-u_i \| < \delta $, the vector
$v_i$ lies in the open $\delta
$-neighbourhood of $P $, hence (i) is satisfied.
Further, for $\delta $ sufficiently small,
the vectors
$v_1, \dots ,v_m$ are also pairwise distinct.

Let $k$ denote the largest integer such that there are arbitrarily small
perturbations
 $B_i$
of our original
points $A_i$
that satisfy the
following:
all edges have the right length
$| B_iB_{i+1} | =S_i$,
and the
dimension of the affine hull of $B_1, \dots ,B_m$ has dimension $k$. 

Assume for contradition that $k<n$.
Then, by $m \ge n+1 \ge k+2$, 
there is an affine dependence among the $B_i$'s. Let,
for example, 
$B_m$ lie in the affine hull $H$ of $B_1, \dots ,B_{m-1}$. Then, fixing
$|B_{m-1}B_m|$ and $|B_mB_1|$, the point 
$B_m$ can move on an $(n-2)$-sphere around the axis $B_{m-1}B_1$ in a
hyperplane perpendicular to $H$. 
 Hence there is an
arbitrarily small perturbation of $B_m$ lying outside $H$, 
while $\aff \{ B_1, \dots ,B_{m-1} \} $
already spans $H$.
Thus 
we have obtained a contradiction to the choice of $k$.

This proves $k=n$ and thus \thetag{ii}. 
\end{proof}
\medskip

%%%%%%%%%%%%%%%%%%%%%%%%%%%%%%%%%%%%%%%%%%%%%%%%%%%%%%%%%%%%%%%%%%%%%%%%%%%%

\subsection {Proof of Theorem~\ref{main3}}
The implication $\thetag {ii} \Longrightarrow \thetag {i}$ is evident.

The implication $\thetag {i} \Longrightarrow \thetag {iii}$ is well-known, but 
we give the proof for completeness.
Using the notations from Theorem F$'$, we have
$S_m = \| S_m u_m \| = \| \sum
_{i=1} ^{m-1} S_i u_i\| \le \sum _{i=1} ^{m-1} S_i$.
The only
case of equality is the degenerate case given in 
condition \thetag {iii$'$} of the
theorem.

Finally, $\thetag {iii} \Longrightarrow \thetag {ii}$ follows from 
Proposition \ref{main} and
Minkowski's Theorem F$'$.

The degenerate case, with \thetag {iii$'$}, follows from the above 
considerations.
\qed 

%%%%%%%%%%%%%%%%%%%%%%%%%%%%%%%%%%%%%%%%%%%%%%%%%%%%%%%%%%%%%%%%%%%%%%%%%%%%%

\subsection {Proofs for Theorem~\ref{main2}}
We need the following relation between the surface area, diameter
and volume of a convex body. Here $\kappa _{n-1}$ is the volume
of the unit ball in ${\mathbb{R}}^{n-1}$.
\begin{prop}[Gritzmann, Wills and Wrase \cite{GWW}]\label{isop} Let
$K\subset \mathbb{R}^n$ be a convex body.
Then the inequality
$S(K)^{n-1} > \kappa _{n-1} \cdot  $\,{\rm{diam}}\,$(K) \cdot 
\left( nV(K) \right) ^{n-2}$
holds and this inequality is sharp.
\qed
\end{prop}

%%%%%%%%%%%%%%%%%%%%%%%%%%%%%%%%%%%%%%%%%%%%%%%%%%%%%%%%%%%%%%%%%%%%%%%%%%%%%%

We will construct the polytope for Theorem~\ref{main2}
by Minkowski's Theorem F$'$.
We need to
choose only an appropriate 
surface area measure.
For a convex polytope, this finite Borel measure is concentrated
in finitely many points. Assume that for given facet areas 
we are far from the degenerate case where
this measure is concentrated in a great-$\mathbb S^{n-2}$. Then by
compactness, the volume of the convex polytope is bounded from
below. Therefore, to get an arbitrarily small volume, we must approach the
degenerate case. This will be done in
the following proof. Recall also the paragraph after Theorem E citing
\cite{IyeIye} (p.~\pageref{thm-E}), where also the
degenerate case was approximated --- however, for simplices only.

%%%%%%%%%%%%%%%%%%%%%%%%%%%%%%%%%%%%%%%%%%%%%%%%%%%%%%%%%%%%%%%%%%%%%%%%%%%%%%

\subsection {
First proof of Theorem~\ref{main2}}
\label{sec:first}

By Proposition \ref{main},
for any $\delta>0$, there are
pairwise distinct vectors
$v_1,v_2, \dots , v_{m} \in \mathbb{S}^{n-1}$
in the open $\delta$-neighbourhood
of the $x_1x_2$-coordinate plane, satisfying the following. They
do not lie in a linear $(n-1)$-subspace of $\mathbb{R}^n$,
and $S_1 v_1+S_2 v_2+ \dots +S_{m}v_{m}=0$.
By Minkowski's Theorem F$'$
there exists a non-degenerate convex polytope $P=P(\delta)$ in $\mathbb{R}^n$
having $m$ facets
with areas
$S_1, \dots, S_{m}$ and unit outer normals $v_1, \dots , v_{m}$.

Let us consider the sequence of polytopes $P_k=P(1/k)$ for $k=1,2,\ldots\,$\,.
We will show that
$V(P_k)\rightarrow 0$ as $k \to \infty$.
Assume the contrary. Then (possibly passing to a subsequence),
we may assume without loss of generality that $V(P_k) \ge \alpha >0$.

By Proposition \ref{isop}, we get the inequality
\begin{equation}
  \nonumber 
S(P_k)^{n-1} > \kappa _{n-1} \cdot \diam(P_k) \cdot (nV(P_k))^{n-2} \ge
\kappa _{n-1} \cdot \diam (P_k) \cdot (n \alpha )^{n-2} ,
\end{equation}
where $S(P_k)=\sum_{i=1}^m S_i$ is a constant.
From this inequality, we conclude that 
$\diam(P_k)$ 
is bounded by some constant $D$ for all~$k$.

By applying translations,
we may assume without loss of generality that
all $P_k$'s have a common point. Therefore, all polytopes $P_k$ lie in a ball
of radius~$D$.
Using compactness (possibly passing to a subsequence),
we may assume even more.
 The sequence $P_k$ tends (in the Hausdorff metric) to a non-empty
compact convex set $P_0$
as $k \to \infty$ \cite[p.~50, Theorem 1.8.6]{Schn}.
Therefore,
$V(P_0)\ge \alpha >0$, and hence $P_0$ is a convex body.
Moreover, by Minkowski's Theorem
G$'$ (actually only by the continuity of the bijection in that
theorem), $P_0$ is a convex polytope having
$m$ facets 
 with all facet outer unit normals in
the $x_1x_2$-coordinate plane. However, this is a contradiction to
condition (ii) of 
Theorem F$'$.
\qed 

%%%%%%%%%%%%%%%%%%%%%%%%%%%%%%%%%%%%%%%%%%%%%%%%%%%%%%%%%%%%%%%%%%%%%%%%%%%%

\begin{remark}
  Instead of Proposition \ref{isop}, where the multiplicative
  constant is sharp, we could have used a consequence of the
  Aleksandrov-Fenchel inequality \cite[p.~327, Theorem 6.3.1]{Schn} to
  show that the diameter is bounded. Namely: the quermassintegrals $W_i(K)$
  (\cite[p.~209]{Schn}) for $0 \le i \le n$ form a logarithmically
  concave sequence. Here, $W_0(K)=V(K)$ and for fixed $n$, $W_1(K)$ is
  proportional to $S(K)$ and $W_{n-1}(K)$ is proportional to the mean
  width of $K$ (\cite[p.~210, p.~291, (5.3.12)]{Schn}).  Then apply
  this logarithmic convexity for volume, constant times surface area
  and constant times mean width. Finally, use the fact that the
  quotient of the diameter and the mean width is between two
  positive numbers (depending only on $n$). This yields the inequality
  of Proposition \ref{isop} with a weaker constant.
\end{remark}

%%%%%%%%%%%%%%%%%%%%%%%%%%%%%%%%%%%%%%%%%%%%%%%%%%%%%%%%%%%%%%%%%%%%%%%%%

\begin{example}\label{needle}
We give an example of a family of tetrahedra ($n=3$ and $m=4$) with
constant facet areas and arbitrarily small volume.
The 
tetrahedra
look like thin vertical needles and have vertices
$( \pm \varepsilon , 0, -(1/ \varepsilon )  \sqrt{ 1- \varepsilon^4/4}
) $ and
$( 0, \pm \varepsilon ,  (1/ \varepsilon)  $
\newline
$\sqrt{ 1- \varepsilon^4/4}
) $.
All facets have area 2, and the volume
is $(4 \varepsilon / 3) \sqrt{ 1- \varepsilon^4/4}$, which tends to
zero
as $\eps\to 0$.
\end{example}

%%%%%%%%%%%%%%%%%%%%%%%%%%%%%%%%%%%%%%%%%%%%%%%%%%%%%%%%%%%%%%%%%%%%%%%%%%%%

\subsection{Second proof of Theorem~\ref{main2}}
\label{sec:second}

{\bf{1.}} First we will construct a partition ${\mathcal P}= \{
P_1, \dots , P_{n+1} \} $ of the index set $\{ 1, \dots ,m \} $
into $n+1$ classes. We will
achieve that the $n+1$ numbers $
\sum_{i \in P_j} S_i$ (for $1
\le j \le n+1$) have the property that
\begin{enumerate}
\narrower
\item [($*$)]
the largest of these numbers 
 is smaller than the sum of all others.
\end{enumerate}
We start with the partition into $m$ singleton classes. Suppose that
we
already have constructed a partition ${\mathcal Q}= \{ Q_1, \dots ,Q_k \} $
such that
\begin{enumerate}
\narrower
\item [($**$)]
the largest of the numbers $T_j:=\sum_{ i \in Q_j}  S_i$,
where $1\le j\le k$,
is smaller than the sum of all the other numbers $T_j$.
\end{enumerate}
If $k=n+1$, then we stop. If $k>n+1$, then
let us
assume
$T_1\le T_2 \le \dots \le T_k$.
Now we take the two classes $Q_1$ and $Q_2$ with
the two smallest sums
and form their union while the other classes $Q_j$ are kept. In
the new partition, the partition class that has maximal sum $T _j$ can
be either the same partition class as in the preceding step or the newly
constructed union. In the first case, ($**$) is
evident. In the second case,
we have $T_1+T_2 \le T_{k-1}+T_k < T_3 +\dots+ T_{k-1}+T_k$ before taking 
the union
since $k \ge n+2 \ge 5$.
Thus $T_{\mathrm{new}} := T_1+T_2 < T_3 +\dots+ T_{k-1}+T_k$.
Therefore,
($**$) holds in this case as well.

This proves that $(*)$ holds for the final partition.

{\bf{2.}} We consider the partition ${\mathcal P}= \{
P_1, \dots ,
P_{n+1}\}$ constructed above. For the sums $R _j:=\sum_{i \in P_j}S_i$, 
we use Theorem E to
construct a non-degenerate
simplex $S$ that has these facet areas and has an arbitrarily
small volume.
Then, for the respective outer unit normals $u_j$ of the facets of this
simplex, we have $\sum _{j=1}^{n+1} R _ju_j=0$.

{\bf{3.}}
We will now split each facet of the simplex into almost parallel
facets
to get the desired polytope with $m$ facets.
Let $\varepsilon >0$ be small. For each $1 \le j \le n+1$,
choose a linear $2$-subspace $X_j$ containing $u_j$.
Choose a
vector $u_{ji}\in X_j$
for
each $i \in P_j$ 
such that
the vectors
$-(1- \varepsilon ) (\sum\limits _{i \in P_j} S_i) u_j$ and $S_i u_{ji}$ (for $i
\in P_j$) are the side
vectors of a convex polygon in $X_j$.
Since the length of the first vector is almost equal to the sum of the others,
 all $u_{ji}$
are close to $u_j$
for $\varepsilon $ sufficiently small. Then all vectors $S_i u_{ji}$ 
for $1 \le j
\le n+1$ and $i \in P_j$ linearly
span ${\mathbb{R}}^n$, and
their sum is 
$$
\sum\limits_{1 \le j \le n+1} \Bigl( \sum\limits_{i \in P_j}
S_iu_{ji}\Bigr)=
\sum\limits_{1 \le j \le n+1} (1-\varepsilon ) 
\Bigl(\sum\limits_{i \in P_j} S_i\Bigr)
u_j=
(1-\varepsilon ) \sum\limits_{1 \le j \le n+1} R _j u_j=0.
$$

{\bf{4.}}
By Minkowski's Theorem F$'$,
there exists a non-degenerate convex polytope with facet
outer unit normals $u_{ji}$ and facet areas $S_i$ (for all $1 \le j
\le n+1$ and all $i \in P_j$).

Observe that we have changed in the course of the proof the surface area measure
only a little bit (in the weak$^*$-topology of
$C(\mathbb{S}^{n-1})^*$). Therefore, after a suitable translation, the
obtained convex polytope is arbitrarily close to the original
simplex $S$ by Theorem~G$'$ (actually only by the
continuity of the inverse of the bijection in that theorem).
Since the simplex had an
arbitrarily small volume,
our convex polytope also has an arbitrarily small volume.
\qed

%%%%%%%%%%%%%%%%%%%%%%%%%%%%%%%%%%%%%%%%%%%%%%%%%%%%%%%%%%%%%%%%%%%%%%%%%%%%%%

\subsection{Third proof of Theorem~\ref{main2}}

The first two 
proofs of Theorem~\ref{main2} did not give geometric
information about the constructed polytopes.
(The first proof used an argument by contradiction and the
second proof used the examples of the simplices.) Now we give a third proof
that is more quantitative and will give also geometric information.
This proof constructs
  ``needle-like'' polytopes, as in Example~\ref{needle}. See also the paragraph 
following the statement of Theorem~\ref{main2} (p.~\pageref{main2}).

First we give the proof 
for $n=3$ dimensions.
We begin with an elementary lemma.  It shows that a convex polytope in
${\mathbb R}^3$ that has steep (almost vertical) facets must have
steep edges, as long as the angles between the normal vectors of
different facets are bounded away from $0$ and $\pi $.

%%%%%%%%%%%%%%%%%%%%%%%%%%%%%%%%%%%%%%%%%%%%%%%%%%%%%%%%%%%%%%%%%%%%%%%%%%%%%%

\begin{lemma}\label{slope}
Consider two planes in ${\mathbb R}^3$ with unit normals
$u^+$ and $u^-$. Assume that $u^+$ and $u^-$
enclose an angle at most $\varepsilon \in (0, \pi /2)$ with the $xy$-plane,
and the angle between them lies in $[\beta , \pi - \beta ]$, where 
$0 < \beta \le \pi /2$.
Then their intersection line encloses an angle at most
\begin{equation}
\nonumber 
\delta := \arcsin \frac{\sin \varepsilon }{\sin (\beta /2)}
\end{equation}
with the $z$-axis, provided that
$\eps \le \beta/2$.
This inequality is sharp.
\end{lemma}

%%%%%%%%%%%%%%%%%%%%%%%%%%%%%%%%%%%%%%%%%%%%%%%%%%%%%%%%%%%%%%%%%%%%%%%%%%

\begin{proof}
We choose a new coordinate system in the following way.
The intersection line becomes
the vertical axis, and the two normal vectors  $u^+,u^- \in
{\mathbb{S}}^2$
lie in the horizontal plane, enclosing an angle $\beta ' \in [
\beta , \pi - \beta ]$ with each other.
 In the new coordinate system, the original
North Pole becomes $n=(n_1,n_2,n_3) \in {\mathbb{S}}^2$.

By hypothesis, 
\begin{equation}\label{a}
\langle n,u^-
\rangle , \langle n,u^+ \rangle \in [- \sin \varepsilon ,  \sin \varepsilon
].
\end{equation}
We want to conclude that 
\begin{equation}
\label{d}
| \langle (0,0,1), n \rangle | = |n_3| \ge \cos \delta ,
\end{equation}
i.e., that
\begin{equation}\label{b}
\sqrt{n_1^2+n_2^2} \le \sin \delta .
\end{equation}
The points $(n_1,n_2) \in {\mathbb{R}}^2$ (projections of $n$ to the
$xy$-plane) for $n$ satisfying (\ref{a}) form
a rhomb of height $2 \sin \varepsilon $ and angles $\beta ', \pi - \beta '
$. A farthest point of this rhomb from $(0,0)$ is one of the
vertices
and its distance from $(0,0)$ is
$\max \{ (\sin \varepsilon )/\sin (\beta '/2),
\allowbreak
(\sin \varepsilon )/ \cos (\beta '/2) \}
\le (\sin \varepsilon )/ \sin ( \beta /2)
= \sin \delta $. That is, (\ref{b}), or equivalently, (\ref{d}) holds and
both are sharp inequalities. Hence,
the inequality of the lemma holds and it is sharp.
\end{proof}

%%%%%%%%%%%%%%%%%%%%%%%%%%%%%%%%%%%%%%%%%%%%%%%%%%%%%%%%%%%%%%%%%%%%%%%%%%%%%%

\begin{lemma}\label{volume}
Consider a convex polyhedron $P \subset {\mathbb R}^3$
with facet areas $S_1, \dots, S_m$.
Assume that its facet outer normals
enclose an angle at most $\varepsilon $ with the $xy$-plane and the
angle between any two of them
 lies in $[\beta , \pi - \beta ]$, where $0 < \beta
\le \pi /2$. Then its volume is bounded by
$$
V(P) \le 2^{-1/4} \pi ^{-1} \cdot \left( \sum _{i=1}^m S_i^{3/4} \right) ^2
\cdot \left( \frac{\sin \varepsilon }{\sin (\beta /2) } \right) ^{1/2},
$$
if $(\sin \varepsilon )/\sin (\beta /2) \le 1/\sqrt{2}$.
\end{lemma}

%%%%%%%%%%%%%%%%%%%%%%%%%%%%%%%%%%%%%%%%%%%%%%%%%%%%%%%%%%%%%%%%%%%%%%%%%%%%%

\begin{proof}
We denote by $s_i(z)$ the length of the horizontal cross-section
of the $i$-th facet at height $z$, and by $s_i^{\max}$ the maximum length of
such 
a horizontal cross-section.  Let $h_i$ be the ``height'' of the $i$-th face:
the difference between the maximum and the minimum
$z$-coordinates of its points.
Let $h_i'$ be the ``tilted height'' of this facet in
its own plane, i.e., the height when the plane is rotated into
vertical position about one of its horizontal cross-sections.

Since $(\sin \varepsilon )/\sin (\beta /2) < 1$, we have by
Lemma~\ref{slope} that $P$ has no horizontal edges.
Therefore, 
using the quantity $\delta $ introduced in
Lemma~\ref{slope}, we get
\begin{equation}\label{uppera}
s_i^{\max} \le h_i \cdot \tan \delta .
\end{equation}
Namely, from the minimal $z$-coordinate --- where
$s_i(z)=0$ --- 
$s_i(z)$ can increase only with a speed at most
$2\tan \delta \,\,( < \infty )$ to reach its maximal value
$s_i^{\max}$. This is clear for a vertical face, and for a nonvertical
face the speed is even smaller. Observe that the $i$-th facet lies in an
upwards circular cone with vertex the lowest point of the $i$-th facet and
directrices enclosing an angle $\delta $ with the $z$-axis. 
From the maximal value it must decrease again
with speed at most $2 \tan \delta $ till
$0$ at the maximal $z$-coordinate.

Therefore, using for (\ref{upperb}) inequality (\ref{uppera}),
\begin{align}
\label{upper}
S_i
&\ge
s_i^{\max} h_i'/2 \ge s_i^{\max} h_i/2
\\
&\ge
(s_i^{\max})^2 /(2\tan \delta).
\label{upperb}
\end{align}
This gives
\begin{equation}
\label{ell}
s_i^{\max} \le \sqrt{2S_i \tan \delta} .
\end{equation}
These relations allow us to bound the volume $V(P)$ as follows, by using the
isoperimetric inequality on each horizontal slice.
\begin{align}
V(P) \nonumber
&=
\int _ {- \infty} ^ {\infty }
(\text{area of cross-section of $P$ at height $z$})\, dz\\\nonumber
&\le
\int_ {- \infty} ^ {\infty }
\left( \sum_{i=1}^m s_i(z) \right) ^2 dz / (4\pi )\\\nonumber
&=
\sum_{i=1}^m
\sum_{j=1}^m \int_ {- \infty} ^ {\infty }
s_i(z)s_j(z) \,dz/(4\pi )\\*\label{intbound}
&\le
\sum_{i=1}^m
\sum_{j=1}^m  s_i^{\max}s_j^{\max} \min\{h_i,h_j\}/(4\pi )\\\nonumber
&\le
\sum_{i=1}^m
\sum_{j=1}^m  s_i^{\max}s_j^{\max} \sqrt{h_ih_j}/(4\pi )\\\nonumber
&=
\biggl( \sum_{i=1}^m s_i^{\max} \sqrt{h_i} \biggr) ^2 /(4\pi )\\\nonumber
&=
\left( \sum_{i=1}^m
\sqrt{  s_i^{\max}}
\sqrt{s_i^{\max}h_i} \right) ^2
\big/ (4\pi )\\\label{use}
&\le
\Bigl( \sum_{i=1}^m
(2S_i \tan \delta ) ^ {1/4}
\sqrt{2S_i} \,\Bigr) ^2
\big/(4\pi )\\\label{use-lemma-1}
&=
\frac{\Bigl( \sum_{i=1}^m
S_i^{3/4}\Bigr) ^2}
{\sqrt2\pi}
\cdot
\sqrt{
 \frac{(\sin \varepsilon )/\sin (\beta/2)}
{\sqrt{1- (\sin ^2 \varepsilon )/ \sin ^2 (\beta/2)}} }
\\
\nonumber
&\le
\frac{\left( \sum_{i=1}^m
S_i^{3/4}\right) ^2}
{\sqrt2\pi}
\cdot
\sqrt{\frac{\sqrt{2} \sin \varepsilon }{\sin (\beta/2)}}\,.
\end{align}
The first inequality uses the isoperimetric inequality.
The second inequality~\eqref{intbound} bounds the integral by an upper 
bound
of the non-negative 
integrand times the length of the interval where the integrand is
positive.
 For~\eqref{use},
 we have used
\eqref{upper} and~\eqref{ell}. To obtain
\eqref{use-lemma-1},
we have used Lemma~\ref{slope}.
The last inequality simplifies the denominator under the assumption
$(\sin \varepsilon )/\sin (\beta /2) \le 1/\sqrt{2}$ of the lemma.
\end{proof}

%%%%%%%%%%%%%%%%%%%%%%%%%%%%%%%%%%%%%%%%%%%%%%%%%%%%%%%%%%%%%%%%%%%%%%%%%%%%%%

\subsection{Third proof of Theorem \ref{main2} for $n=3$ dimensions}
As in the first proof, we use Minkowski's Theorem F$'$.
We want to apply Lemma~\ref{volume}, making $\eps$ small.
Thus, we must let the normal vectors with given
lengths $S_i$ converge to the $xy$-plane, keeping their sum to be
$0$. Moreover,
the linear span of the outer unit facet normals should be ${\mathbb{R}}^3$.
Then we apply Minkowski's Theorem F$'$.
In the limiting configuration
the normals will lie in the $xy$-plane. They must form angles in 
$[\beta , \pi - \beta ]$ (with 
$\beta \in (0, \pi /2]$) with each other in order that 
Lemma~\ref{volume} should work.
Thus we must avoid parallel sides.

Will show that there is only one exceptional case in which parallel sides 
cannot be avoided.
Consider a planar convex
$m$-gon $M$ with sides $S_i$ that has the minimum number of parallel pairs of 
sides.
Let us assume that $M$ 
has a side such
that the sum of the two incident angles is different from $\pi $. Then
by a small length-preserving motion of this side and the neighbouring two sides,
one can achieve the following.
This side changes its direction while new parallel pairs
of sides are not created.
Therefore, $M$
can have a parallel pair of sides only if, for each
of these sides, the sum of the incident angles is $\pi $. That is,
we have four vertices that determine two parallel sides and 
whose outer angles (i.e., $\pi $ minus
the inner angles) have sum
$2 \pi $. Since the sum of all outer angles is $2 \pi
$, there are no more vertices and
$M$ 
must be a parallelogram. 
If its sides are not equal then we rearrange the side vectors so as to
obtain a (convex) deltoid that is not a parallelogram.
So the only remaining case is when
$m=4$ and $S_1=S_2=S_3=S_4$.
However, this case has been treated in Example~\ref{needle}:
a tetrahedron with four faces of equal areas and having an arbitrarily small
positive volume.
Suitable inflations provide examples for
all values of $S_i$.

Disregarding this exceptional case,
 we have now a strictly convex polygon in the $xy$-plane without parallel sides.
Assume that the angle between any two edges is in the range 
$[\beta_1,\pi-\beta_1]$ for some $\beta_1>0$.
We still need to perturb the sides so that the edge vectors
span ${\mathbb{R}}^3$.
Consider the first three consecutive vertices $A_1,A_2,A_3$ of $M$. Let us fix
$A_1$ and $A_3$. Rotate
the two sides $[A_1,A_2]$ and $[A_2,A_3]$ about the line through $A_1$ and
$A_3$ through a small angle $\alpha>0$ while keeping their lengths fixed.
The $m-2\ge2$ remaining side vectors span the $xy$-plane since they are not
parallel. At the same time, the vector ${\overrightarrow{A_1A_2}}$ 
points out of the $xy$-plane
and therefore the edge vectors span $\mathbb{R}^3$.

By making the angle of 
rotation $\alpha$ small enough, we can ensure the following.
The angle
between all edge vectors of the perturbed polygon $M(\alpha)$ is still in the 
range
$[\beta_2,\pi-\beta_2]$ for some fixed $\beta_2>0$. Moreover,
the angle $\varepsilon$
of the side vectors
with the $xy$-plane can be made arbitrarily small.
We use the edge vectors $\overrightarrow{S_i}$
of $M(\alpha)$ as outer normals and construct
the polytope $P$ by Minkowski's
Theorem F$'$, with $S_i := \| {\overrightarrow{S_i}} \|$ and
$u_i :={\overrightarrow{S_i}}/S_i$. 
By Lemma~\ref{volume}, the volume can be made arbitrarily small.
\qed
\medskip

\begin{remark}
In the polytope
that we have constructed, all facets except
two are vertical.
By going through the proof of~Lemma~\ref{volume}, one can
see the following.
It would have been sufficient to assume the
constraint $[\beta,\pi-\beta]$ on the angles for those pairs of facet
normals that involve at least one of
the two nonvertical facets.
\end{remark}

%%%%%%%%%%%%%%%%%%%%%%%%%%%%%%%%%%%%%%%%%%%%%%%%%%%%%%%%%%%%%%%%%%%%%%%%%%%%%%

\begin{example}
For odd dimension $n=2k+1$,
there is a higher-dimensional generalization of  Example~\ref{needle}.
Consider a large regular
$k$-simplex of edge length $a:=1/\eps$ in the
$x_{k+2} \ldots x_n$-coordinate plane. It has $k+1$ vertices
$v_1,\dots,v_{k+1}$. At each vertex $v_i$, we draw a short segment of
length $b:=\eps$
centred at $v_i$
in the direction of the $x_{i}$-axis. 
The convex hull of the union of
these segments
is an $n$-simplex with congruent facets. The facet areas are
$\sim \mathrm{const} \cdot  a^kb^k=\const $,
while the volume is $ \mathrm{const} \cdot a^kb^{k+1}
=\const\cdot \eps$, which becomes arbitrarily small as $\eps\to0$.
\end{example}

%%%%%%%%%%%%%%%%%%%%%%%%%%%%%%%%%%%%%%%%%%%%%%%%%%%%%%%%%%%%%%%%%%%%%%%%%%%%%%

\subsection{
Third proof of Theorem~\ref{main2} for $\bf {n>3}$ dimensions}

\subsubsection{Construction of an almost flat spatial polygon}
\label{spatial}

As for $n=3$, we start with a planar convex $m$-gon $M$ 
in the $x_1x_2$-coordinate
plane, where  $m\ge n+1$. It has 
side vectors ${\overrightarrow{S_i}}$ (this notation will be preserved
also after perturbations)
with side lengths $S_i$, for
$1 \le i \le m $.

$M$ is contained in the $x_1 \ldots x_{n-1}$-coordinate hyperplane $X$.
By  small perturbations of the closed
polygon $M$ in $X$ that preserve the side lengths~$S_i$, we want to
achieve that
$$
\begin{minipage}{0.75\textwidth}
\raggedright
the perturbed (skew) closed polygon $M \subset X$  has no
$n-1$ side vectors lying in an 
 $(n-2)$-dimensional 
linear subspace of $X$.
\end{minipage}
\leqno{(*)}
$$

Initially,  $M$ lies in a 2-dimensional plane.
We will fulfill~$(*)$ by following the proof of Proposition \ref{main}.
Our desired conclusion is slightly stronger than in
Proposition~\ref{main}: there we excluded only the case that \emph{all}
vectors lie in a lower-dimensional subspace.

Assume that some $i\le n-1$ side vectors lie in a linear subspace of dimension 
less than
$i$, where $i$ is the smallest number with this property.
We will eliminate these linear dependencies iteratively.
We have already seen how we can avoid parallel edges ($i=2$).
The only case where parallel sides could not be avoided was 
$m=4$ and $S_1=S_2=S_3=S_4$ (a rhomb) and this happens only for $n \le m-1=3$.
Therefore, we
can assume $i \ge 3$.
 Observe also that
by small
perturbations the different vertices remain distinct.

We may any time rearrange the cyclic order of side vectors of $M$ as we want.
So we assume that the first $i$ side vectors
${\overrightarrow{S_1}}, \dots ,{\overrightarrow{S_i}}$ are linearly
dependent,
and
 any $i-1$ side vectors are linearly independent.
Number the vertices $A_j$ so that ${\overrightarrow{S_j}}$ 
goes from vertex $A_j$ to $A_{j+1}$
(indices taken modulo $m$).
We want remove this linear dependence by perturbing the vertex $A_{i}$.
For a technical reason, we have to first refine the order of the side
vectors even further.
Let
$ \sum _{j=1}^{i} \lambda _j  
{\overrightarrow{S_j}}=0$ be the (unique, up to a scalar factor)
linear dependence.
The $\lambda _j$'s cannot be all equal, since
$ \sum _{j=1}^{i}
{\overrightarrow{S_j}} =
 {\overrightarrow{A_1A_{i+1}}}
=0$ would imply that $A_1$ and $A_{i+1}$ are equal points.
The polygon has $m>n>i$ sides, and hence this is excluded.
Therefore, by permuting the side vectors if necessary, we
can assume that $\lambda_{i-1}\ne\lambda_i$. 
Since all ($i-1$)-subsets are linearly independent,
we have $\lambda_j\ne 0$ for all $j$, and thus
 we can assume
without loss of generality that
$\lambda_i=-1$.
In other words,
$ {\overrightarrow{S_i}} = \sum _{j=1}^{i-1} \lambda _j  
{\overrightarrow{S_j}}$,
with $\lambda _{i-1} \ne -1$. 

Now, fixing $A_{i-1}$, $A_{i+1}$, $\| {\overrightarrow{S_{i-1}}} \| $, and $\|  
{\overrightarrow{S_i}} \| $, we can perturb $A_i$ as follows.
The point $A_i$ moves on 
an $(n-3)$-dimensional
sphere in a hyperplane within $X$ with affine hull orthogonal to the segment
$[A_{i-1},A_{i+1}]$.
There is a small 
motion that moves $A_i$ out of the 
 subspace
$H := \aff\{ A_1,A_2,\dots , A_{i-1},
A_{i+1} \} $ of $X$. 

Now we show that the dimension of this subspace $H$ is in fact
$i-1$, which implies that the dimension of
 $ \aff\{ A_1,A_2,\dots,
A_i,A_{i+1} \} $ increases by $1$ and
${\overrightarrow{S_1}}, \dots ,{\overrightarrow{S_i}}$ become linearly
independent.
Clearly, $\dim H$ cannot be greater than $i-1 = \dim
 \aff\{ A_1,A_2,\dots,
A_{i-1},A_{i+1} \} $.
To see that
$\dim H = i-1$, we note that
 the vectors 
$ {\overrightarrow{A_1A_2}} = {\overrightarrow{S_1}}$,
${\overrightarrow{A_2A_3}} = {\overrightarrow{S_2}}, \dots , 
{\overrightarrow{A_{i-2}A_{i-1}}} = {\overrightarrow{S_{i-2}}}$, 
${\overrightarrow{A_{i-1}A_{i+1}}} =
{\overrightarrow{S_{i-1}}} + \sum _{j=1}^{i-1} \lambda _j
{\overrightarrow{S_j}}$ are linearly independent,
since $\lambda _{i-1} \ne -1$,
 and hence, their linear span 
has dimension $i-1$.

Thus, we have established that, by perturbing $A_i$, the vectors
${\overrightarrow{S_1}}, \dots ,{\overrightarrow{S_i}}$ become linearly
independent. If the perturbation is small enough, then
every set of side vectors 
that was
linearly independent
before the motion remains linearly independent.
Therefore, the number of linearly dependent
$i$-tuples of side vectors of $M$ 
decreases.
A finite number of iterations 
eliminates all linearly dependent
$i$-tuples, and $i$ can be increased (till $n-1$), until
\thetag{$*$} is eventually established.
This concludes the construction of the polygon~$M$.

Condition~\thetag{$*$} can be rephrased in the following way.
The
determinant of any $n-1$ normed side vectors
${\overrightarrow{S_i}}/S_i$
of $M$ (i.e., the signed volume of the parallelepiped spanned by them)
is nonzero.
We denote by $b>0$ the smallest
absolute value of these determinants. This bound will play the
role of the sine of the angle bound $\beta$ in
Lemmas~\ref{slope} and~\ref{volume}.

%%%%%%%%%%%%%%%%%%%%%%%%%%%%%%%%%%%%%%%%%%%%%%%%%%%%%%%%%%%%%%%%%%%%%%%%%%%

\subsubsection{Steep facets imply steep edges}

We generalize Lemma~{\ref{slope}} to higher dimensions:

\begin{lemma}\label{slope1}
Let $n\ge 3$, and consider $n-1$ hyperplanes in $\mathbb{R}^n$ making
an angle at most $\eps < \pi/2$ with the vertical axis \textup(the
$x_n$-axis\textup).
If their unit normal vectors $v_1, \dots , v_{n-1}$
span an $(n-1)$-parallelotope of volume
at least $b\,\,(>0)$ then they intersect in a line. The angle
between this
line and the vertical direction is bounded by
$$ \delta := \arcsin \frac{(n-1)^{3/2} \sin\eps}b,$$
provided that ${(n-1)^{3/2} \sin\eps}\le b$.

For fixed $n$, the order of magnitude of this bound $\delta$ as a
function of $\varepsilon $ and $b $ is optimal.
More precisely, for any $\eps$ and $b$, where $0<\eps<\pi/2$ and $0<b \le 1$,
there are instances with
$\sin \delta =\min\{1,(\sin \varepsilon ) / \sin ((\arcsin b) /2)\}$.
\end{lemma}

\begin{proof}
Since the unit normal vectors $v_1,\dots,v_{n-1}$ are linearly
independent,
the intersection of the hyperplanes is a line $\ell$.
Let us choose a new orthonormal coordinate system where $\ell$ is the last
coordinate axis. Then the last coordinate of the vectors $v_i$ is
zero, and
we may write these vectors as $v_i=\binom{v_i'}0$ with $v_i'
\in\mathbb{S}^{n-2}
\subset\mathbb{R}^{n-1}$.
By assumption, the $(n-1)\times (n-1)$ matrix $V=(v_1',\dots,v_{n-1}')$
has determinant of absolute value $\lvert \det V \rvert \ge b$.

Let $p=\binom{p'}{p_n}$, with $p'
\in\mathbb{R}^{n-1}$,
be the unit vector of the original positive $x_n$-direction
in the new coordinate system.
Its angle $\delta \in [0, \pi /2]$ with the line $\ell$ 
satisfies
$\cos \delta = \lvert  p_n \rvert$
and $\sin\delta = \lVert p'\rVert$,
and thus our goal is to show that
\begin{equation}
  \label{eq:goal}
 \lVert p'\rVert \le (n-1)^{3/2}\frac{
 \sin\eps}{b}\,.
\end{equation}
Let $\alpha_i$ denote the angle between
$p$ and the normal $v_i$.
By the angle assumption on the hyperplanes,
we have
$\pi/2-\eps\le  \alpha_i\le\pi/2+\eps$. Therefore,
with
 $r_i := \cos \alpha_i = \langle p,v_i\rangle
 = \langle p',v_i'\rangle$,
we have $\lvert r_i\rvert \le \sin\eps$.

The $n-1$ equations  $\langle p',v_i'\rangle =r_i$ form a linear
system
$(p')^T V = (r_1,\dots,r_{n-1})$ (the column vectors of $V$ being the
$v_i$'s), i.e., $V^Tp'=(r_1,\dots , r_{n-1})^T$, 
which
determines
$p'$ uniquely:
\begin{equation}
  \label{eq:inverse}
p' = (V^T)^{-1}(r_1,\dots,r_{n-1})^T \,.
\end{equation}
We write 
$\mathop{\mathrm{adj}} (V^T)$ for the transpose of the matrix whose
entries are the
signed cofactors of the respective entries of $V^T$.
By the formula $(V^T)^{-1}=\mathop{\mathrm{adj}} (V^T)/\det (V^T)$, 
each entry
of $(V^T)^{-1}$ is an $(n-2)\times (n-2)$ subdeterminant of
$V^T$ divided by $\pm \det V^T$.
The rows of the submatrices of
$V^T$ are vectors of length at most $1$ and therefore
these subdeterminants are bounded in absolute value by $1$.
It follows that the entries of $(V^T)^{-1}$ 
are bounded in absolute value by $1/b$.
Since the $|r_i|$'s are at most
$\sin\eps$, we get from
\eqref{eq:inverse} that the $n-1$ entries of $p'$ are bounded by
$(n-1)(\sin\eps)/b$
in absolute value. Hence we have proved~\eqref{eq:goal}.

To establish the lower bound, we can lift the tight three-dimensional
example from Lemma~{\ref{slope}} to $n$ dimensions. For $\varepsilon \le \beta
/2$,
the enclosed angle will be the same as in three dimensions, namely
$\arcsin ( (\sin \varepsilon )/  \sin ( \beta /2) ) $,
where $\sin\beta = b$. 
For 
$\varepsilon > \beta /2$, we use the example with $\varepsilon = \beta /2$.
We embed the 3-dimensional example
into ${\mathbb{R}}^n$
by a linear isometry that maps the positive $x,y,z$-axes to the positive
$x_1,x_2,x_n$-coordinate axes of ${\mathbb{R}}^n$.
(The ``vertical'' direction is now the direction of the $x_n$-axis.)
The two $2$-planes 
of the three-dimensional example are turned into hyperplanes as follows.
We
replace them by their inverse images under the orthogonal projection of 
${\mathbb{R}}^n$ to the $x_1x_2x_n$-coordinate
subspace. Simultaneously, we add
the hyperplanes with equations $x_3=0, \dots ,x_{n-1}=0$.
\end{proof}

%%%%%%%%%%%%%%%%%%%%%%%%%%%%%%%%%%%%%%%%%%%%%%%%%%%%%%%%%%%%%%%%%%%%%%%%%%%%%

\subsubsection{Polytopes with steep facets have small volume}

\begin{lemma}\label{volume1}
Let $n > 3$ be an integer. Assume that a convex polytope $P \subset
{\mathbb R}^n$
has facet areas $S_1,\ldots,S_m$. Moreover, its outer unit facet normals
enclose an angle at most $\varepsilon \in (0, \pi /2)$
with the $x_1 \ldots x_{n-1}$-plane. Also
the volume of the $(n-1)$-parallelepiped spanned by
any $n-1$ unit facet normals of $P$
is at least $b>0$.
Then its volume 
is bounded by
$$
V(P) \le {\rm{const}}_n \cdot \Bigl( \sum_{i=1}^m S_i ^{n/(2n-2)} \Bigr) ^2
\cdot \left( \frac{\sin \varepsilon}b \right) ^{1/(n-1)},
$$
if \,$\sin ^2 \varepsilon \le b^2/[2(n-1)^3]$.

On the other hand,
for $n\ge 3$ and any $m \ge 2n$,
there exists a suitable \,$\varepsilon _0 \in (0, \pi /4)$, 
such that the following holds. For any \,$\varepsilon \in
(0, \varepsilon _0)$, there exists a convex polytope $P(\varepsilon ) \subset
{\mathbb{R}}^n$, with $m$ facets, with the following properties. It
satisfies the hypotheses of this lemma \textup(except the one about the facet
areas\textup), with $b$ depending only
on $n$ and $m$, such that
\begin{align*}
V \left( P(\varepsilon ) \right) &\ge {\text{\rm{const}}} _n' \cdot
S \left( P(\varepsilon ) \right) ^{n/(n-1)} \cdot (\tan \varepsilon )^{1/(n-1)}
\\
&\ge  {\text{\rm{const}}} _n' \cdot m^{-(n-2)/(n-1)} 
\cdot \left( \sum _{i=1} ^m S_i(\varepsilon )
 ^{n/(2n-2)}
\right) ^2 \cdot (\tan \varepsilon )^{1/(n-1)} .
\end{align*}
Here, $S_1(\varepsilon ), \ldots ,S_m (\varepsilon )$ 
are the areas of the facets of $P(\varepsilon )$.
In particular, in the inequalities of Lemma~\ref{volume} and this lemma,
the order of magnitude as a function of $\varepsilon $ is optimal.
\end{lemma}

%%%%%%%%%%%%%%%%%%%%%%%%%%%%%%%%%%%%%%%%%%%%%%%%%%%%%%%%%%%%%%%%%%%%%%%%%%%

\begin{proof}
We begin with the proof of the upper estimate.
We denote by $s_i(x_n)$ the $(n-2)$-volume of the horizontal cross-section
of the $i$-th facet at height $x_n$. Moreover, we denote 
by $s_i^{\max}$ the maximum $(n-2)$-volume of such 
a horizontal cross-section.  Let $h_i$ be the ``height'' of the $i$-th facet:
the difference between the maximal and the minimal
$x_n$-coordinates of its points.
Let $h_i'$ be the ``tilted height'' of this facet in
its own hyperplane. That is, the height when the hyperplane is rotated into
vertical position about one of its horizontal cross-sections.

Now, since $\sin ^2 \varepsilon \le b^2/[2(n-1)^3]$,
the angle $\delta $ from
Lemma~\ref{slope1} lies in $(0, \pi /2)$. Hence, by Lemma~\ref{slope1}, $P$
has no horizontal edges, and thus, also no horizontal $k$-faces 
for any $k \in \{
1, \ldots ,n-2 \} $. Therefore, once more by
Lemma~\ref{slope1}, we know that 
every facet is contained in two
rotationally symmetric cones with $(n-1)$-balls as bases.
One cone
has its apex at the unique lowest point of this
facet
and extends upwards from there.
Its axis is vertical (parallel to the $x_n$-direction),
and the directrices enclose an angle 
$\delta $ with the $x_n$-axis.
The other cone extends downwards from the 
highest point of the facet and has a vertical axis and directrices enclosing an 
angle $\delta $ with the $x_n$-axis.
We use the upwards cone 
from the minimal height
till the arithmetic mean of the minimal and maximal heights. We use
the downward cone 
for the other half of the vertical extent of the facet.
By this argument, we can bound
the maximum cross-section area 
$s_i^{\max}$
of the $i$-th facet as follows.
\begin{equation}
\label{lower}
s_i^{\max} \le \left( (h_i/2) \cdot \tan \delta \right) ^{n-2} 
\cdot \kappa _{n-2} \,.
\end{equation}
(From the minimal height
till the arithmetic mean of the minimal and maximal heights we have the
following.
Any horizontal cross-section of the cone is contained in some
$(n-1)$-ball of radius at most 
$R := (h_i/2) \cdot \tan \delta$. Thus, any horizontal
cross-section
of the facet lies inside the intersection of its own affine hull 
with the upwards cone. 
That is, it lies in the intersection of an $(n-2)$-dimensional affine subspace
with a cone whose base is
an $(n-1)$-ball of radius at most $R$. 
Hence, this horizontal cross-section lies
inside some $(n-2)$-ball of radius at most~$R$. A similar argument holds
for the downward cone.)
Moreover, we also have
\begin{equation}
\label{upper1}
S_i \ge
s_i^{\max} h_i'/(n-1) \ge
s_i^{\max} h_i/(n-1) .
\end{equation}
Let us rewrite \eqref{lower} and \eqref{upper1} as follows.
\begin{align}
 h_i^{-(n-2)}\cdot  s_i^{\max}
 &\le
\left( (\tan \delta )/2 \right) ^{n-2} \cdot \kappa _{n-2} 
\label{A}\\
 h_i\cdot s_i^{\max}  
&\le 
(n-1)S_i .
\label{B}
\end{align}
We multiply the
$1/[(2n-2)(n-2)]$-th power
of \eqref{A} with the $
n/(2n-2)$-th power of \eqref{B} to get an inequality that we will need.
\begin{multline}
\label{eq:needed}
(s_i^{\max})^{(n-1)/(2n-4)}
\sqrt{h_i}
\\
\le
\left( (\tan \delta )/2 \right) ^{1/(2n-2)} \cdot (\kappa _{n-2})
^{1/[(2n-2)(n-2)]}
\cdot ((n-1)S_i)^{n/(2n-2)} .
\end{multline}
Let $K := [(n-1)^{n-1} \kappa _{n-1} ] ^{-1/(n-2)}$ denote the
constant of the isoperimetric inequality in $n-1$ dimensions:
\begin{equation}
  \label{eq:iso}
V_{n-1}(C)\le K\cdot
(V_{n-2}(\partial C))^{(n-1)/(n-2)}  
\end{equation}
(for $C \subset {\mathbb{R}}^{n-1}$). Now we can bound the volume as follows.
\begin{align} \nonumber
V(P)
&=
\int _ {- \infty} ^ {\infty }
\bigl[\text{$(n-1)$-volume of the cross-section of $P$ at height 
$x_n$}\bigr]\, dx_n
\\ \nonumber
& \le
\int_ {- \infty} ^ {\infty }
\left[ \biggl( \sum_{i=1}^m s_i(x_n) \biggr) ^{(n-1)/(2n-4)} \right]^2 dx_n
\cdot K
\\ \nonumber
& \le
\int_ {- \infty} ^ {\infty }
\left[  \sum_{i=1}^m s_i(x_n) ^{(n-1)/(2n-4)}  \right]^2 dx_n
\cdot K
\\ \nonumber
&=\int_ {- \infty} ^ {\infty }
\sum_{i=1}^m
\sum_{j=1}^m 
s_i(x_n)^{(n-1)/(2n-4)}s_j(x_n)^{(n-1)/(2n-4)} \,dx_n \cdot K
\\ \nonumber
&=
\sum_{i=1}^m
\sum_{j=1}^m \int_ {- \infty} ^ {\infty }
s_i(x_n)^{(n-1)/(2n-4)}s_j(x_n)^{(n-1)/(2n-4)} \,dx_n \cdot K
\\ \label{bound-interval}
&\le
\sum_{i=1}^m
\sum_{j=1}^m  (s_i^{\max})^{(n-1)/(2n-4)}(s_j^{\max})^{(n-1)/(2n-4)}
\min\{h_i,h_j\}\cdot K
\\ \nonumber
&\le
\sum_{i=1}^m
\sum_{j=1}^m  (s_i^{\max})^{(n-1)/(2n-4)}(s_j^{\max})^{(n-1)/(2n-4)}
\sqrt{h_ih_j} \cdot K
\\
&\le\label{use-prepared}
({\tan \delta}) ^{1/(n-1)} \cdot
2  ^{- 1/(n-1)} \cdot
(\kappa _{n-2}) ^{1/[(n-1)(n-2)]}\cdot
(n-1)^{n/(n-1)}
\\ \nonumber
&\qquad\times
\biggl(
\sum_{i=1}^m
S_i^{n/(2n-2)} \biggr) ^2 
\cdot K
\\
&= \label{use-lemma3}
\const_n \cdot 
\biggl(
\sum_{i=1}^m
S_i^{n/(2n-2)} \biggr) ^2 
\biggl( \frac{(n-1)^{3/2} (\sin \varepsilon )/b}
{\sqrt{1- (n-1)^3 (\sin ^2 \varepsilon )/ b ^2}} \biggr) ^{1/(n-1)} 
\\ \nonumber
&\le
\const'_n \cdot 
\biggl(
\sum_{i=1}^m
S_i^{n/(2n-2)} \biggr) ^2 
\cdot \left( \frac{\sin \varepsilon}b \right) ^{1/(n-1)}\,.
\end{align}
The first inequality uses the isoperimetric inequality~\eqref{eq:iso}. The 
second inequality
uses the concavity of the function $t^{(n-1)/(2n-4)}$ for $t \in [0, \infty )$
and its vanishing at $t=0$. (Observe
that $0 < (n-1)/(2n-4) \le 1$.)
Inequality \eqref{bound-interval}, as in
\eqref{intbound},
bounds the integral of a non-negative function
by an upper bound of the integrand times
the length of the interval where the integrand is
positive.
For
\eqref{use-prepared}, we have used the bound
\eqref{eq:needed} that we derived above.
Inequality
\eqref{use-lemma3} uses the bound $\delta$ from
Lemma~\ref{slope1}.
Finally, by hypothesis, the expression under the square root
in the denominator of~\eqref{use-lemma3}
is bounded below by
${1- (n-1)^3 (\sin ^2 \varepsilon )/b^2} \ge 1/{2}$.
We have therefore established the claimed upper bound.
\medbreak

Now we give the example for the lower bound for $n \ge 3$ and $m \ge 2n$.
Let $\varepsilon \in (0, \varepsilon _0 )$, where $\varepsilon _0 \in (0, \pi
/4)$ will be
chosen later. Let us write ${\mathbb R}^n={\mathbb R}
^{n-1} \oplus {\mathbb R}$. Let $T^+, T^- \subset {\mathbb{R}}^{n-1}$
be regular $(n-1)$-simplices circumscribed about
the unit ball $B^{n-1}$ of ${\mathbb R}^{n-1}$. Put them in such a general
position w.r.t.\ each other 
so that
any $n-1$ of
 their altogether $2n$ facet outer normals
linearly
span ${\mathbb{R}}^{n-1}$. 
Let $n \le m^+,m^-$ and $m=m^++m^-$. Let $R^{\pm }$ be obtained from $T^{\pm
}$ by intersecting it still with $m^{\pm }-n$ closed halfspaces in 
${\mathbb{R}}^{n-1}$, all containing $B^{n-1}$, with their boundaries touching
$B^{n-1}$. Then 
$$
B^{n-1} \subset R^{\pm } \subset T^{\pm } \subset (n-1)B^{n-1}.
$$
Let the altogether $m=m^++m^-$ facet outer unit normals of $R^+$ and $R^-$
satisfy the same condition of general position as above. Namely,
any $n-1$ of them linearly
span ${\mathbb{R}}^{n-1}$. 
Let $b>0$ be the minimum of the $(n-1)$-volumes of
the $(n-1)$-parallelotopes spanned by any $n-1$ of these altogether $m$
facet outer unit 
normals. 

Observe that for $n=3$ and $m \ge 2$,
the largest value of $b$ is $\sin ( \pi /m)$ --- if we do not begin the
construction with two regular triangles but allow any $m$ 
facet outer unit normals in ${\mathbb{S}}^{n-2} = {\mathbb{S}}^1$. 
For $n>3$, the maximal value of $b$ can be bounded
from above as follows --- again not beginning with two regular simplices,
but allowing any $m$ facet outer unit normals in ${\mathbb{S}}^{n-2}$. 
Let us choose altogether
$n-1$ outer unit normal vectors of $R^+$ and $R^-$, 
say, $u_1, \ldots , u_{n-1} \in {\mathbb{S}}^{n-2}$. 
We have 
\begin{align*}
\frac{|{\text{det}}\,(u_1, \ldots , u_{n-1})|}{(n-1)!} &= 
V_{n-2} ({\text{conv}}
\, \{ u_1, \ldots , u_{n-1} \} ) \cdot {\text{dist}}\,(0, {\text{aff}}\,
\{ u_1, \ldots , u_{n-1} \} ) /(n-1) 
\\
&\le
V_{n-2} ({\text{conv}} \, \{ u_1, \ldots , u_{n-1} \} )/(n-1).
\end{align*}
Here, $\text{dist}(\cdot,\cdot)$ denotes distance.
Thus, it suffices to bound $V_{n-2} ({\text{conv}}
\, \{ u_1, \dots , u_{n-1} \} )$ from above. This is the spherical analogue
--- for ${\mathbb{S}}^{n-2}$ ---
of the celebrated Heilbronn problem. This problem asks about the maximum of the
minimal $n$-volume of
$n$-simplices spanned by any $m$ points in $[0,1]^n$. This problem is poorly
understood. For an extensive
literature on this problem, see \cite[Ch.~11.2]{BrMoPa}.
Unfortunately, this spherical
variant cannot be reduced to the case of $[0,1]^{n-2}$ by taking the projection
of, say, the intersection of ${\mathbb{S}}^{n-2}$ with each orthant
to the tangent ${\mathbb{R}}^{n-2}$ at its centre. Namely,
the area of 
${\text{conv}} \, \{ u_1, \ldots , u_{n-1} \} $ can be large even if its
projection has a small area. In one direction, we have an
implication: large projection areas imply large areas --- but
large areas
still do not imply large values of $|{\text{det}}\,(u_1, \ldots , u_{n-1})|$.
However, this spherical variant 
is a special case of the $(n-1)$-dimensional 
Heilbronn problem for $[0,1]^{n-1}$. Namely, we can
just add to any set of $(n-1)$-dimensional vectors in ${\mathbb{S}}^{n-2}$
the single vector $0$ --- but probably
 we loose an essential  part of the information in this way.

Let $P^{\pm } (\varepsilon )$ 
be the half-infinite pyramid with vertex $(0, \ldots , 0, \pm
\tan \varepsilon )$ and base $R^{\pm }$. Then 
$$
C_i ^{\pm } (\varepsilon )
\subset P^{\pm } (\varepsilon )
\subset C_o ^{\pm }(\varepsilon ),
$$
where $C_i ^{\pm }(\varepsilon )$ and $C_o ^{\pm }(\varepsilon )$ 
is a half-infinite cone with
vertex $(0, \ldots , 0, \pm \tan \varepsilon )$ and base
$B^{n-1}$ and $(n-1) B^{n-1}$, respectively. Therefore,
$$
C_i (\varepsilon ) := C_i ^+ (\varepsilon ) \cap C_i ^-(\varepsilon ) \subset
P(\varepsilon ): = P^+(\varepsilon ) \cap P^-(\varepsilon ) \subset
C_o(\varepsilon ) := C_o ^+(\varepsilon ) \cap C_o ^-(\varepsilon ).
$$
Here, $C_i(\varepsilon )$ and $C_o(\varepsilon )$ are
double cones over $B^{n-1}$ and $(n-1)B^{n-1}$, respectively, with vertices
$(0, \ldots , 0, \pm \tan \varepsilon )$. Moreover, $P(\varepsilon )$ 
is a convex polytope
with $m$ facets, all facet outer unit normals enclosing an angle $\varepsilon
$ with the $x_1 \ldots x_{n-1}$-hyperplane. (Actually they enclose 
an angle $\varepsilon $
with the respective facet
outer unit normal of $R^+$ or $R^-$ in ${\mathbb{R}}^{n-1}$.) 
If \,$\varepsilon _0$ and thus also
$\varepsilon $ is
sufficiently small then still any $n-1$ facet outer unit normals of
$P(\varepsilon )$ span
an $(n-1)$-parallelotope of volume at least some $b' \in (0,b)$.

A routine calculation gives
$$
\frac{V\left( P(\varepsilon ) \right) }{S\left( P(\varepsilon ) \right)
^{n/(n-1)}} \ge 
\frac{V\left( C_i(\varepsilon ) \right) }{S\left( C_o(\varepsilon ) \right)
^{n/(n-1)}}
=
\frac{(\tan \varepsilon ) ^{1/(n-1)}} {n (2 \kappa _{n-1}) ^{1/(n-1)} 
[ 1+(n-1)\tan ^2 \varepsilon ] ^{n/(2n-2) }}.
$$
Therefore,
\begin{align*}
V \left( P(\varepsilon ) \right)
& \ge 
\frac{
S \left( P(\varepsilon ) \right) ^{n/(n-1)}
\cdot (\tan \varepsilon ) ^{1/(n-1)}}{
 n (2 \kappa _{n-1}) ^{1/(n-1)} 
[ 1+(n-1)\tan ^2 \varepsilon _0] ^{n/(2n-2)}}
\\
&\ge 
\frac{
m^{-(n-2)/(n-1)} \left( \sum _{i=1}^m S_i (\varepsilon ) 
^{n/(2n-2)} \right) ^2 \cdot 
(\tan \varepsilon ) ^{1/(n-1)}} 
{n (2 \kappa _{n-1}) ^{1/(n-1)} 
[ 1+(n-1)\tan ^2 \varepsilon _0] ^{n/(2n-2) }}. 
\end{align*}
Here, $S_1(\varepsilon ), \dots , S_m(\varepsilon )$ 
are the areas of the facets of $P(\varepsilon )$.
The second inequality is equivalent to H\"older's inequality for the 
numbers $S_i(\varepsilon )$, between their arithmetic mean and their power
mean with exponent $n/(2n-2) \in (0,1)$. Finally, 
observe that $\tan \varepsilon _0 \in
(0,1)$.
\end{proof}

%%%%%%%%%%%%%%%%%%%%%%%%%%%%%%%%%%%%%%%%%%%%%%%%%%%%%%%%%%%%%%%%%%%%%%%%%

\subsubsection{Conclusion of the proof}
Now we can finish the third proof of Theorem~\ref{main2} for $n>3$. We 
proceed as
for $n=3$ but instead of Lemma~\ref{volume} we use
Lemma~\ref{volume1}.
In Section~\ref{spatial} (see its last paragraph),
we have constructed a closed $m$-gon $M$ in the $(n-1)$-dimensional
subspace $X$ with the following property. 
Any $n-1$ normed side vectors ${\overrightarrow{S_i}}
/S_i$ span a parallelotope
of volume at least $b$.
We follow the third proof of Theorem~\ref{main2} for $n=3$. 
We take the first three consecutive vertices $A_1,A_2,A_3$ of $M$ and
``rotate''
$A_2$ out of the subspace $X$, keeping $A_1,A_3$ and the lengths 
$|A_1A_2|$ and $|A_2A_3|$ fixed.
We have a whole $(n-2)$-dimensional sphere on which $A_2$ can move, which
intersects $X$ orthogonally.
By bounding the distance by which $A_2$ moves by a suitable threshold
we can ensure the following.
Any $n-1$ normed side vectors ${\overrightarrow{S_i}}
/S_i$ still span a parallelotope
of volume at least $b'$ with some weaker bound $b'>0$.
The angle between $[A_1,A_2]$ or $[A_2,A_3]$ and the ``horizontal''
hyperplane $X$ can be made arbitrarily small.
Thus, Lemma~\ref{volume1} guarantees that the volume tends to zero as well.
\qed

%%%%%%%%%%%%%%%%%%%%%%%%%%%%%%%%%%%%%%%%%%%%%%%%%%%%%%%%%%%%%%%%%%%%%%%%%%%%
%%%%%%%%%%%%%%%%%%%%%%%%%%%%%%%%%%%%%%%%%%%%%%%%%%%%%%%%%%%%%%%%%%%%%%%%%%%%
%%%%%%%%%%%%%%%%%%%%%%%%%%%%%%%%%%%%%%%%%%%%%%%%%%%%%%%%%%%%%%%%%%%%%%%%%%%%

\section{Proofs for the hyperbolic case}
For general concepts in hyperbolic geometry, we refer to \cite{AVS,Ba,
Cox,Li,Milnor,Pe}. In particular, a {\it{Lambert quadrilateral in}}
${\mathbb{H}}^2$ is a quadrilateral that has three right angles.

%%%%%%%%%%%%%%%%%%%%%%%%%%%%%%%%%%%%%%%%%%%%%%%%%%%%%%%%%%%%%%%%%%%%%%%%%%%%%%

\subsection{Proof of Proposition \ref{necne}}
Let $H$ be the hyperplane of the facet $F_m$ of $P$ of area $S_m$.
Let $p\colon {\mathbb{H}}^n \to H$
be the orthogonal projection of ${\mathbb{H}}^n$ to $H$.
 The image by $p$ of the union of
the $m-1$ facets different from $F_m$ contains $F_m$.

Let $dS$ be a surface element at a point $x \in {\mathbb{H}}^n$. Let its
image by $p$ be the surface element $dS'$ at $p(x)$. Clearly, it suffices
to show that $dS' \le dS$. We may assume that $dS$ is an (infinitesimal)
$(n-1)$-ball of radius $dr$
in the tangent space $T_x({\mathbb{H}}^n)$ of ${\mathbb{H}}^n$ at
$x$.

First we deal with the case when $dS$ is orthogonal to the line $\ell \left(
x,p(x) \right) $. (For $x \in H$, we mean by $\ell \left(
x,p(x) \right) $
the line containing $x$ and orthogonal
to $H$.)
Then $dS'$ is an infinitesimal $(n-1)$-ball in $T_{p(x)}({\mathbb{H}}^n)$ of
some radius $dr'$.
 By the trigonometric formulas
of Lambert quadrilaterals in ${\mathbb{H}}^2$ 
 (see \cite[\S 29, (V)]{Pe}
or \cite[Theorem  2.3.1]{Buser}),
we have
$1 \le \cosh |xp(x)| = \left( \tanh(dr) \right) /\tanh(dr')$.
Hence, $dr' \le dr$ and therefore, $dS' \le dS$.

Now we extend this analysis to the case when $dS$ is not orthogonal to the line
$\ell \left( x,p(x) \right)$.
Then the image by $p$ of the infinitesimal $(n-1)$-ball $dS$
in $T_x({\mathbb{H}}^n)$ is an infinitesimal $(n-1)$-ellipsoid in
$T_{p(x)}({\mathbb{H}}^n)$. It  has $n-2$ semiaxes equal to $dr'$
and the $(n-1)$-st semiaxis smaller than $dr'$.
Hence, $dS'<dS$ in this case.

The case of equality is clear: the polytope must degenerate to the
doubly counted facet $F_m$.
\qed 

%%%%%%%%%%%%%%%%%%%%%%%%%%%%%%%%%%%%%%%%%%%%%%%%%%%%%%%%%%%%%%%%%%%%%%%%%%%%%

\subsection{Proof of Proposition \ref{necne'}}
From maximality of $S_m,S_{m-1}, \dots ,
S_3$, it follows that all
vertices lie at infinity. Namely, a vertex cannot be incident only to the
facets of areas $S_2,S_1$. Hence, also $S_2,S_1$ are maximal.
\qed 

%%%%%%%%%%%%%%%%%%%%%%%%%%%%%%%%%%%%%%%%%%%%%%%%%%%%%%%%%%%%%%%%%%%%%%%%%%%%

\subsection{Proof of Proposition \ref{nonexhyp}}
Let $P$ be a convex polyhedron as in the proposition with respective facets
$F_1, \dots , F_m$. Let us consider any vertex $v$ of some facet $F_i$.
In the facets incident to $v$, the
angle of $F_i$ at $v$ is at most the sum of the angles of all
other facets incident to $v$. 
To see this, we 
intersect $P$ with an infinitesimally small sphere with centre at this
vertex (in the conformal model).  We obtain a convex spherical polygon whose
side lengths are the (convex) 
angles of the facets incident to $v$ at $v$,
all these angles being in $[0, \pi )$.

Summing these inequalities over all vertices $v$ of $F_i$ we obtain the 
following.
The sum $t_1$ of the angles of $F_i$ is at
most the sum $t_2$
of the angles of all other facets 
 at the vertices of $F_i$.
The sum $t_2$ is bounded above by the sum $t_3$
of \emph{all angles of all 
facets different
from} $F_i$.  The resulting
inequality $t_1 \le t_3$ is equivalent to the inequality to be proved.

Clearly, if we have at least one finite vertex with incident edges not in a
plane, then we have at least one strict
inequality among the summed inequalities. So in this case,
we have strict inequality in the proposition.
\qed 

%%%%%%%%%%%%%%%%%%%%%%%%%%%%%%%%%%%%%%%%%%%%%%%%%%%%%%%%%%%%%%%%%%%%%%%%%%%%%

The inequality $t_1 \le t_3$ from this proof
is discussed for the spaces ${\mathbb{R}}^3$ and $\mathbb{S}^3$ in
Remark~\ref{t1-t3} in Section~\ref{6.1}.

%%%%%%%%%%%%%%%%%%%%%%%%%%%%%%%%%%%%%%%%%%%%%%%%%%%%%%%%%%%%%%%%%%%%%%%%%%%

\subsection{Proof of Theorem~\ref{exhyp}}
\begin{prop}[{\cite[p.~127]{AVS}, \cite[Theorem 1, Proposition 2]{HM}}]
\label{maxv}
For $n \geq 2$, a simplex in $\mathbb{H}^n$ \textup(with vertices at infinity
admitted\textup) is of maximal volume if and
only if all its vertices are at infinity and it
is regular. It has a finite volume.
\end{prop}

%%%%%%%%%%%%%%%%%%%%%%%%%%%%%%%%%%%%%%%%%%%%%%%%%%%%%%%%%%%%%%%%%%%%%%%%%%%%%

Let $v_n$ be the maximal volume of a simplex in $\mathbb{H}^n$.
For instance, $v_2=\pi$ and $v_3= -3\int_0^{\pi/3}\log|2\sin u|\,du
= 1.0149416\ldots $
(\cite[p.~20]{Milnor}, \cite[p.~127]{AVS}).
Obviously,
 the facet areas $S_i$ of a compact simplex in
$\mathbb{H}^n$
are smaller than $v_{n-1}$.

%%%%%%%%%%%%%%%%%%%%%%%%%%%%%%%%%%%%%%%%%%%%%%%%%%%%%%%%%%%%%%%%%%%%%%%%%%%%%

\begin{lemma}\label{hs1}
  The area $S$ of a right triangle $\Delta ABC \subset \mathbb{H}^2$ 
   with angle $\angle ACB=\pi/2$ and side lengths $|AC|=b$ and
  $|BC|=a$ 
  fulfills the equation
$$
\tan S =\frac{\sinh a\cdot \sinh b }{\cosh a +\cosh b } \,.
$$
\end{lemma}
\begin{proof}
This is a routine consequence of the trigonometric formulas for a right
triangle in~$\mathbb{H}^2$. We use $S= \pi /2 -\alpha - \beta $,
$\tan\angle CBA=(\tanh b) /\sinh a $,
$\tan\angle CAB=(\tanh a) /\sinh b $
 \cite[p.~238]{Cox}, and $\tanh x = (\sinh x) /\cosh x $.
\end{proof}

%%%%%%%%%%%%%%%%%%%%%%%%%%%%%%%%%%%%%%%%%%%%%%%%%%%%%%%%%%%%%%%%%%%%%%%%%%%%%

\begin{lemma}\label{hs4}
Let $d>0$.
Assume that $\Delta ABC \subset \mathbb{H}^2$ is a triangle such that
$|AB|\leq d$ and $|AC|\leq d$.
Then the area $S$ of this triangle is bounded by the inequality
$$
S \leq 2\arctan \frac{\cosh d -1}{2\sqrt{\cosh d }} \,.
$$
\end{lemma}

%%%%%%%%%%%%%%%%%%%%%%%%%%%%%%%%%%%%%%%%%%%%%%%%%%%%%%%%%%%%%%%%%%%%%%%%%%%%%

\begin{proof}
Without loss of generality we may assume that $|AB|=|AC|=d$.
Let $H$ be the orthogonal projection of $A$ to the line $\ell (B,C)$.
The segment $AH$ cuts the triangle $ABC$ into two congruent right triangles.
With $x=\cosh |AH|$ and
$y=\cosh |BH|=\cosh |CH|$, we have $x,y\geq 1$
and $xy=\cosh d $.
Let $S$ be the area of $\Delta ABC$. Lemma~\ref{hs1} gives
$$
\tan^2(S/2)=\frac{(x^2-1)(y^2-1)}{(x+y)^2} \,.
$$
Looking for the maximum of the numerator and the minimum of the
denominator
subject to the constraints $x,y > 0$ and $xy=\cosh d $,
we see that the maximal value of $S$ is attained for
$x=y=\sqrt{\cosh d }$. This proves the lemma.
\end{proof}

%%%%%%%%%%%%%%%%%%%%%%%%%%%%%%%%%%%%%%%%%%%%%%%%%%%%%%%%%%%%%%%%%%%%%%%%%%%%%

To show that a tetrahedron with given facet areas exists, we will
use a topological argument, which is encapsulated in
the following lemma. The lemma guarantees the existence of a zero of a
function under certain conditions on the boundary.
\begin{lemma}\label{index}
Let $F=(f_1,f_2)\colon P\rightarrow \mathbb{R}^2$ be a continuous function
defined on a rectangular domain $P=[0,a]\times [0,b]$, where $a,b>0$.
 Assume that there are
$u_1,u_2,v_1,v_2\in \mathbb{R}$ such that
$u_1>u_2$ and $v_1<v_2$ and
\begin{gather*}
f_1(x,0)+f_2(x,0)\leq 0,  \quad
f_1(x,b)+f_2(x,b)\geq 0, {\text{\,\,\,\,and}} \\
u_1f_1(0,y)+u_2f_2(0,y)\leq 0,  \quad
v_1f_1(a,y)+v_2f_2(a,y)\leq 0
\end{gather*}
for every $0\leq x\leq a$ and every $0\leq y \leq b$. Then there exists a point
$(c,d)\in P$ such that $F(c,d)=(0,0)$.
\end{lemma}

%%%%%%%%%%%%%%%%%%%%%%%%%%%%%%%%%%%%%%%%%%%%%%%%%%%%%%%%%%%%%%%%%%%%%%%%%%%%%

\begin{proof}
  The conclusion clearly holds if $F$ vanishes at some point of the
  boundary $\partial P$ of $P$. If $F$ has no zero on $\partial P$,
  then it is sufficient to establish that the index of the vector field $F$
  on the curve $\partial P$ is $1$. This implies that $F$ has a zero
  in the interior of $P$ 
  \cite[p.~98, proof of Theorem VI.12, sufficiency]{HW}.

To determine the index of $F$, we define the 
auxiliary function 
$F_0\colon \partial P \to {\mathbb{S}}^1$ as follows. 
On the vertical boundaries of $P$, we let
 $F_0 (0,y) = A_L := (-1/ \sqrt{2}, 1/ \sqrt{2}) $ 
and
 $F_0 (a,y) = A_R := -A_L = (1/ \sqrt{2}, -1/ \sqrt{2}) $ 
for  $0\leq y \leq b$.
On the lower boundary,
$F_0(x,0)=(\xi,\eta)$ turns counterclockwise in the
 half-plane $\xi+\eta \le 0$
with constant angular velocity from
$A_L$ to $A_R$ as $x$ varies from $0$ to $b$.
 The upper boundary is similar, but there
$F_0(x,b)$ changes
clockwise in the
 half-plane $\xi+\eta \ge 0$.
Then it follows from the assumptions that, 
for $(x,y) \in \partial P$, $F(x,y)$ and
$F_0(x,y)$ never point to opposite directions.
 Hence, $F(x,y)/\| F(x,y) \| ,
F_0(x,y) \colon \partial P \to {\mathbb{S}}^1$ are homotopic. Therefore,
the index of
$F$ equals the index of $F_0$, namely~$1$.
\end{proof}

%%%%%%%%%%%%%%%%%%%%%%%%%%%%%%%%%%%%%%%%%%%%%%%%%%%%%%%%%%%%%%%%%%%%%%%%%%%%%

We still need two lemmas that together form a sharpening of
two lemmas from~\cite{BKM}.

\begin{lemma}[{\cite[Lemmas 1 and 2]{BKM}}]\label{BKM1}
Consider a \textup(possibly degenerate\textup) triangle $A$ in 
${\mathbb{S}}^2$, ${\mathbb{R}}^2$ or
${\mathbb{H}}^2$ with sides $a,b,x$, where $a,b>0$. For the case
of $\mathbb{S}^2$,
we additionally assume $a+b \le \pi $. Then, for $a,b$ fixed
and $|a-b| \le x \le a+b$, the area $\overline A$ of this triangle
is a concave function of
$x$. \textup(For $x=a+b=\pi$ on $\mathbb{S}^2$, we define $\overline A$ 
by a limit 
procedure: namely, fixing $a,b$, we let $x \to a+b= \pi $. Accordingly, we
set $\overline A = \pi $. Observe that for $a+b= \pi$, the area $\overline{A}$
is half the area of a digon with sides containing 
the sides $a,b$ of $A$.\textup) 
In addition, the area is strictly concave for
${\mathbb{R}}^2$ and ${\mathbb{H}}^2$ and, under the 
additional constraint $a+b < \pi $, also for
${\mathbb{S}}^2$.
\qed 
\end{lemma}

%%%%%%%%%%%%%%%%%%%%%%%%%%%%%%%%%%%%%%%%%%%%%%%%%%%%%%%%%%%%%%%%%%%%%%%%%%%%%

We calculate more precise details about this concave function and the
value of its maximum.

\begin{lemma}\label{BKM2}
We use the notations and hypotheses
of Lemma~\ref{BKM1} and denote by $\gamma $ the angle 
between 
the
sides $a,b$. For ${\mathbb{S}}^2$, let us additionally assume $a+b < \pi $.
Then $\overline A$ equals $0$ for $x=|a-b|$ and $x=a+b$, and it has a unique
maximum for some value $x=x_{\max}$, with corresponding angle 
$\gamma = \gamma _{\max}$.

For ${\mathbb{H}}^2$, we have
\begin{gather*}
\cosh (x_{\max}/2)=\sqrt{(\cosh a + \cosh b)/2},
\\
\cos \gamma _{\max} =  \tanh (a/2) \cdot \tanh (b/2), 
\end{gather*}
and the
value of the maximal area is
\begin{align*}
&\pi - 2 \arcsin \frac{  \sinh  (a/2) }{\sinh r} -
2 \arccos \frac{ \tanh ( a/2)}{ \tanh r} + {}\\
&\pi - 2 \arcsin \frac{  \sinh ( b/2) }{\sinh r} -
2 \arccos \frac{ \tanh  (b/2)}{ \tanh r} ,
\end{align*}
where $\cosh  r = \sqrt{(\cosh a +\cosh b)/2}$.

 For\, ${\mathbb{R}}^2$, 
we have
\[
x_{\max}^2=a^2+b^2, \ 
 \gamma _{\max} = \pi/2, 
\]
and the maximal area is
$ab/2$.

For \,${\mathbb{S}}^2$, we have
\begin{gather*}
  \cos (x_{\max}/2)=\sqrt{(\cos a + \cos b)/2},
\\
\cos \gamma _{\max} = - \tan (a/2) \cdot \tan (b/2),
\end{gather*}
and the maximal area is
\begin{align*}
&
2 \arcsin \frac{  \sin (a/2)}{ \sin r } +
2 \arccos \frac{  \tan (a/2)}{ \tan r } 
- \pi + {}\\
{}{} &
2 \arcsin \frac{  \sin (b/2)}{ \sin r } +
2 \arccos \frac{  \tan (b/2)}{ \tan r } - \pi ,
\end{align*}
where $\cos  r = \sqrt{(\cos a +\cos b)/2}$.

Moreover,
letting $y/2$ be
the distance between the midpoint of the side $x$ and the common vertex of the 
sides $a$ and $b$, we have the following equivalences:
$$ 
\gamma \in [0, \gamma _{\max}) \Longleftrightarrow x < y, {\text{ and }}
\gamma = \gamma _{\max} \Longleftrightarrow x = y, {\text{ and }}
\gamma \in (\gamma _{\max}, \pi ] \Longleftrightarrow x > y.
$$
\end{lemma}

%%%%%%%%%%%%%%%%%%%%%%%%%%%%%%%%%%%%%%%%%%%%%%%%%%%%%%%%%%%%%%%%%%%%%%%%%%%%%

\begin{proof}
For ${\mathbb{R}}^2$, the statement is elementary. Therefore, we investigate
only the cases of ${\mathbb{H}}^2$ and ${\mathbb{S}}^2$.

Denote the vertices of the triangle opposite to the sides
$a,b$ and $x$, by $A,B$
and $C$, respectively. Let $D$ be the mirror image of $C$ with respect to the
midpoint of the side $x$. Then the quadrilateral $ABCD$ is centrally symmetric
with respect to the intersection $O$ of its diagonals $BC$ (of length $x$)
and $AD$ (of length $y$). Its area is $2{\overline A}$, so it suffices to
investigate its area.

We recall the isoperimetric property of the circle in
${\mathbb{R}}^2$, ${\mathbb{H}}^2$, and on ${\mathbb{S}}^2$
 --- but in the last case of radius $r < (a+b)/2 < \pi /2$ --- 
among sets of equal
perimeter. Namely, that the maximum area is attained for the circle.
For ${\mathbb{S}}^2$, one must restrict the candidate to (closed)
sets contained in 
some open
half-${\mathbb{S}}^2$
\cite[Ch.~18, \S 6, {\bf{2}}, (18.39)]{San}.  Observe that a piecewise
$C^1$ closed curve on ${\mathbb{S}}^2$ with length less than $2 \pi $
lies in some open half-${\mathbb{S}}^2$, by elementary
integral-geometric considerations \cite[Ch. 7, \S 2, (7.11) and
Ch. 18, \S 6, {\bf{1}}, (18.37)]{San}.  (A very detailed exposition of
the isoperimetric inequality in spaces of constant curvature, i.e., in
${\mathbb{R}}^n$, ${\mathbb{H}}^n$, and ${\mathbb{S}}^n$, can be found
in~\cite{Schm1}. See~\cite{Schm2} for further details.)

For $\gamma = 0$, we have $x=|a-b| < a+b=y$, while
for $\gamma = \pi $, we have $x=a+b > |a-b|=y$.
Therefore, for some $\gamma \in (0, \pi )$, we have $x=y$. This implies that,
for this $\gamma $, i.e., for this $x$, 
$ABCD$ is inscribed in a circle of
centre $O$ and radius $r:=x/2=y/2$. By the isoperimetric
property of the circle --- on ${\mathbb{S}}^2$ of radius 
$r < (a+b)/2 < \pi /2$,
in the sense described above ---
 this value of $ \gamma $ must therefore be $ \gamma _{\max}$, and
this $x$ is $x_{\max}$, see
\cite[p.~63, Problem 21]{Ka}, \cite[\S 5, Problem 63]{JB}, 
\cite[p.~52]{Kr}. (These references deal with the case of ${\mathbb{R}}^2$.
However,
their well-known proof carries over to ${\mathbb{H}}^2$ and ${\mathbb{S}}^2$
if we use the isoperimetric property of the circle --- 
on ${\mathbb{S}}^2$ of radius $r < (a+b)/2 < \pi /2$,
in the sense described above.)

We determine the
radius of this circle. We use the law of cosines for the triangles
$\Delta AOC,\Delta BOC$ and write $\varphi := \angle BOC$.
 For
${\mathbb{H}}^2$, we have
$$
\cosh a  = \cosh (x/2) \cdot \cosh (y/2) - \sinh (x/2) \cdot \sinh (y/2) \cdot
\cos \varphi ,
$$
and
$$
\cosh b  = \cosh (x/2) \cdot \cosh (y/2) + \sinh (x/2) \cdot \sinh (y/2) \cdot
\cos \varphi .
$$
Adding these, we obtain
\begin{equation}\label{PARLAW1}
\cosh a + \cosh b = 2 \cosh (x/2) \cdot \cosh (y/2) .
\end{equation}
Analogously, for ${\mathbb{S}}^2$, we obtain
\begin{equation}\nonumber
\cos a  + \cos b  = 2 \cos (x/2) \cdot \cos (y/2) .
\end{equation}
(These are the analogues of the parallelogram law in ${\mathbb{R}}^2$.)
Thus, for ${\mathbb{H}}^2$, we have
\begin{equation}\nonumber
\cosh a  + \cosh b  = 2 \cosh^2 r = 2 \cosh ^2 (x_{\max}/2) ,
\end{equation}
and, for ${\mathbb{S}}^2$, we have
\begin{equation}\nonumber
\cos a  + \cos b  = 2 \cos ^2 r = 2 \cos ^2 (x_{\max}/2) 
\end{equation}
in the range $0<r < (a+b)/2 < \pi /2$.

Furthermore, for ${\mathbb{H}}^2$, 
$x$ is a strictly increasing function of $\gamma $
and, by (\ref{PARLAW1}), $y$ is a strictly 
decreasing function of
$x$. Hence, 
for $x=2r$ we have $y=2r$,
for $|a-b| \le x < 2r$ we have $2r < y \le a+b$,
and similarly, 
for $2r < x \le a+b$ we have $|a-b| \le y < 2r$. 
These imply the last equivalences in the lemma for ${\mathbb{H}}^2$.

Next we determine $\cos \gamma _{\max} $ for ${\mathbb{H}}^2$. The law
of cosines for the triangle $\Delta ABC$ gives
$$
\cos \gamma _{\max} = \frac{ \cosh a  \cdot \cosh b  - \cosh (2r)}
{\sinh a  \cdot \sinh b } .
$$
In this equation, we have
$\cosh (2r) = 2 \cosh ^2 r - 1 = \cosh a  + \cosh b  - 1 $, and this
implies the formula in the lemma. (Observe that $0 < a,b$ implies $\cos \gamma
_{\text{max}} = \tanh (a/2) \cdot \tanh (b/2) \in (0,1)$.)

For ${\mathbb{S}} ^2$, the proof of the
last equivalences in the lemma and the calculation of $\cos
\gamma _{\max}$
are analogous. (Observe that now $0 < a,b$ and $a/2+b/2 < \pi /2$ imply 
$\cos \gamma _{\text{max}} = 
- \tan (a/2) \cdot \tan (b/2) \in (-1,0) $.)

Finally, the value of the maximum follows from the trigonometric formulas for a
right triangle in ${\mathbb{H}} ^2$ and ${\mathbb{S}} ^2$.
\end{proof}

%%%%%%%%%%%%%%%%%%%%%%%%%%%%%%%%%%%%%%%%%%%%%%%%%%%%%%%%%%%%%%%%%%%%%%%%%%%%%

We prove Theorem~\ref{exhyp} with the following

{\bf{Construction 1.}}
Consider a number 
\begin{equation}\label{S}
S \in (0,\pi/2)
\end{equation} 
and a number $t>0$ such that
\begin{equation}\nonumber 
2{\text{sinh\,}}(t/2) > \tan S.
\end{equation}
(Later, $S$ will be the area of a compact right triangle, which
explains 
condition~(\ref{S}).
At the same time, this explains the hypothesis
$0 < S_4 < \pi /2$ of Theorem~\ref{exhyp},
since in the proof, $S$ will be chosen for example as $S_4$.)

Now we define a function
\begin{equation}\nonumber
f_{t,S}\colon [0,t]\rightarrow \mathbb{R}
\end{equation}
as follows. For any $x\in [0,t]$,
consider the function
$$
g_x(y):=\arctan \frac{\sinh x \cdot \sinh y }{\cosh x +\cosh y }+
\arctan \frac{\sinh(t-x) \cdot \sinh y }{\cosh(t-x)+\cosh y },
$$
where $y \in [0, \infty )$.
It is easy to see that $(d/dy)g_x(y)>0$
for $y\in [0,\infty)$, and
$g_x(0)=0$, and
\begin{align*}
\lim_{y\to \infty} g_x(y)
&=
\arctan(\sinh x)+\arctan(\sinh(t-x))
\\
&\geq
\arctan \left( \sinh x  + \sinh (t-x) \right) \ge \arctan(2\sinh(t/2))>S
\end{align*}
for all $x \in [0,t]$. Here, at the first inequality, we used concavity of the
function $\arctan y$ on $[0, \infty ) $ and $\arctan 0 =0$.
At the second inequality, we used convexity of the
function sinh\,$x$ on the interval $[0,t]$.
Therefore, there is a unique $\tilde{y} \in (0,\infty)$
such that $g_x(\tilde{y})=S$. We put
\begin{equation}\label{tilde y}
f_{t,S}(x):=\tilde{y} \in (0,\infty).
\end{equation}
Now we investigate some properties of this function.
Obviously, $f_{t,S}$ is continuous on $[0,t]$ (moreover, it is $C^1$  on
$(0,t)$), and
$f_{t,S}(x)=f_{t,S}(t-x)$.

Here is the geometric interpretation of $f_{t,S}$.
Consider a triangle $\Delta ABC \subset \mathbb{H}^2$ with the
following properties:
\begin{enumerate}
\item  $|AB|=t$,
\item  the area of $\Delta ABC$ is $S$,
\item if $H$ is the orthogonal projection of $C$ to the line $\ell
(A,B)$, then $H \in [A,B]$ and $|AH|=x \in [0,t]$.
\end{enumerate}
Then it is easy to see
(using
 Lemma~\ref{hs1}) that 
 $|CH|=f_{t,S}(x)$.
It is also easy to see that for $0 < \tilde{S}<S$
and for every $x\in [0,t]$, we have
$f_{t,\tilde{S}}(x)<f_{t,S}(x)$.

In what follows we determine
the number
\begin{equation}\label{htS}
h_{t,S}:=f_{t,S}(0)=f_{t,S}(t).
\end{equation}
By Lemma \ref{hs1}, we
have
$$
\tan S = \frac{\sinh t \cdot\sinh h_{t,S}}{\cosh t +\cosh h_{t,S}}.
$$
Solving this equation for $\cosh h_{t,S}$, we get
$$
\cosh h_{t,S}=\frac{\tan^2 S\cdot\cosh t +\sqrt{1+\tan^2 S}\cdot\sinh^2t}
{\sinh^2t-\tan^2S} \,.
$$
(Observe that $\sinh t > 2 \sinh (t/2) > \tan S$ by the strict convexity of the
function $\sinh t$ on $[0, \infty )$ and $\sinh 0 = 0$. Therefore, the
denominator in this formula is positive.)
From this, we see that $\cosh h_{t,S} \to 1/\cos S $ for $t \to \infty $.

%%%%%%%%%%%%%%%%%%%%%%%%%%%%%%%%%%%%%%%%%%%%%%%%%%%%%%%%%%%%%%%%%%%%%%%%%%%%

\begin{proof}[Proof of Theorem~\ref{exhyp}]
{\bf{1.}}
First we consider the case when hypothesis~(\ref{fc1}) of Theorem~\ref{exhyp} 
holds.

Let us take a $t>0$ such that
\begin{equation}\label{tS_4}
2\sinh(t/2)> \tan S_4,
\end{equation}
and
for the number $h_{t,S_4}:=f_{t,S_4}(0)=f_{t,S_4}(t)$ 
defined in (\ref{tilde y}) and
(\ref{htS}),
we have
\begin{equation}\label{tdef}
\frac{\cosh h_{t,S_4}-1}{2\sqrt{\cosh h_{t,S_4}}} < \tan (S_1/2).
\end{equation}
Such a $t$ exists since $\cosh h_{t,S_4} \to 1/\cos S_4$ for $t \to
\infty$, and
$$
\lim\limits_{t\to \infty}\frac{\cosh h_{t,S_4}-1}{2\sqrt{\cosh h_{t,S_4}}}
=\frac{1-\cos S_4}{2\sqrt{\cos S_4}} <
\tan(S_1/2),
$$
by hypothesis (1) of the theorem.

Consider any plane $\sigma$ in $\mathbb{H}^3$. Take points $A_1,A_2 \in
\sigma$ such that $|A_1A_2|=t$, where $t$
has been chosen above. Let $\sigma^+$ and $\sigma^-$ be the two half-planes
bounded by the line $\ell (A_1,A_2)$ in $\sigma $.

For a given $x\in [0,t]$, we take the
point $H=H(x)$ on the segment $[A_1,A_2]$
satisfying $|A_1H|=x$. Consider the half-line
$l_x$ from $H$ in $\sigma^+$ that is orthogonal to the line $\ell (A_1,A_2)$.
Now let $\tilde{A}_4=\tilde{A}_4(x)$ be the
point on $l_x$ satisfying $|\tilde{A}_4H|
=f_{t,S_3}(x)$,
and let $\tilde{A}_3=\tilde{A}_3(x)$ be the point on $l_x$ satisfying
$|\tilde{A}_3H|
=f_{t,S_4}(x)$.
By $S_4 \geq S_3$, we have $\tilde{A}_4\in[\tilde{A}_3,H]$, and thus also
$|H{\tilde{A}}_4| \le |H{\tilde{A}}_3|$.

Let $\{ \gamma(\varphi) \}$ 
be the one-parameter group of rotations about the
line $\ell (A_1,A_2)$ 
through the angles $\varphi $, in some definite sense of rotation, 
with $\gamma (0)$ being the identity.
For $\varphi \in [0, \pi]$ we consider the point $A_3=A_3(x,\varphi)
:=\gamma(\varphi)(\tilde{A}_3(x))$.
Note that $A_3(x,\pi)\in \sigma^-$. Consider also $A_4(x,\varphi)
:=\tilde{A}_4(x)$, $A_1(x,\varphi):=A_1$ and $A_2(x,\varphi):=A_2$.

We are going to prove that there exists an $(x,\varphi)\in [0,t]\times
[0,\pi]$ 
such that the 
 (possibly degenerate)
tetrahedron $T=T(x, \varphi):=A_1(x,\varphi)\allowbreak A_2(x,\varphi)
\allowbreak
A_3(x,\varphi)\allowbreak A_4(x,\varphi)$ has
facet areas $S_1,S_2,S_3,S_4$, where $S_i$ is the area of the facet 
opposite to $A_i(x,\varphi)$.

Let $s_i(x,\varphi)$ be the area of the facet of $T(x, \varphi)$  that is
opposite to the vertex $A_i=A_i(x,\varphi)$.
By our construction, we obviously have
$s_4(x,\varphi)=S_4$ and $s_3(x,\varphi)=S_3$.

Let us define functions $f_1,f_2\colon [0,t]\times[0,\pi]\rightarrow 
\mathbb{R}$ as
follows.
\begin{equation}\label{ff}
f_1(x,\varphi)=s_2(x,\varphi)-S_2, \quad f_2(x,\varphi)=s_1(x,\varphi)-S_1.
\end{equation}
It is easy to see that
\begin{align}\label{eqq1}
f_1(x,0)+f_2(x,0)&=S_4-S_1-S_2-S_3 < 0,
\\\label{eqq2}
f_1(x,\pi)+f_2(x,\pi)&=S_3+S_4-S_1-S_2 \geq 0.
\end{align}
Now we check that
\begin{equation}\label{eqq3}
f_1(0,\varphi)< 0, \quad f_2(t,\varphi) < 0,
\end{equation}
for all $\varphi \in [0,\pi]$.

For the first inequality in (\ref{eqq3}),
we note that the point $\tilde{A}_3(0)$ satisfies
$|\tilde{A}_3(0)A_1|=f_{t,S_4}(0)=h_{t,S_4}$,
and the point $\tilde{A}_4(0)$ satisfies $|\tilde{A}_4(0)A_1|=f_{t,S_3}(0)
=: h_{t,S_3}\leq h_{t,S_4}$.
Therefore, for every $\varphi \in [0,\pi]$, the triangle
$\Delta A_1A_3(0,\varphi)A_4(0,\varphi)$ satisfies $|A_1A_3(0,\varphi)| =
h_{t,S_4}$ and
$|A_1A_4(0,\varphi)|\leq h_{t,S_4}$. By Lemma~\ref{hs4} and (\ref{tdef}),
we get
$$
s_2(0,\varphi)\leq 2\arctan  \frac{\cosh h_{t,S_4} -1}
{2\sqrt{\cosh h_{t,S_4}}} < S_1 \leq S_2.
$$
Therefore, $f_1(0,\varphi)=s_2(0,\varphi)-S_2 \le s_2(0,\varphi)-S_1<0$
for all $\varphi \in [0,\pi]$.

For the second inequality of (\ref{eqq3}),
we replace in the above argument $A_1,\tilde{A}_3(0)$, and $\tilde{A}_4(0)$
by $A_2,\tilde{A}_3(t)$, and $\tilde{A}_4(t)$, respectively. We get
$f_2(t,\varphi)=s_1(t,\varphi)-S_1<0$ for all $\varphi \in [0,\pi]$.

Taking into account the inequalities (\ref{eqq1}--\ref{eqq3}),
by applying Lemma~\ref{index}
with $(u_1,u_2,$
\newline
$v_1,v_2):=(1,0,0,1)$,
we find an $(x,\varphi)\in [0,t]\times [0,\pi]$ such that
$f_1(x,\varphi)=f_2(x,\varphi)=0$. This means that $s_1(x,\varphi)=S_1$ and
$s_2(x,\varphi)=S_2$
for the corresponding (possibly degenerate) tetrahedron $T$.

{\bf{2.}}
Now we consider the case when hypothesis~(\ref{fc2}) of Theorem~\ref{exhyp}
holds. We use the same
construction of the tetrahedron $T$ as in the first case.
For the functions $f_1$ and $f_2$ defined by (\ref{ff}), we get
the inequalities
(\ref{eqq1}) and (\ref{eqq2}).
Now we check that
\begin{equation}\label{eqq4}
f_2(0,\varphi)\geq 0, \quad f_1(t,\varphi) \geq 0
\end{equation}
for all $\varphi \in [0,\pi]$.
For this we note that
\begin{equation}\label{s_1}
s_1(0,\varphi)\geq s_1(0,0)=S_4-S_3 \geq S_2 \ge S_1 
\end{equation}
and
\begin{equation}\label{s_2}
s_2(t,\varphi)\geq s_2(t,0)=S_4-S_3 \geq S_2 
\end{equation}
for all $\varphi \in [0,\pi]$
provided $t$ is sufficiently large, as we will prove. Of course, we have to
prove only the first inequalities in (\ref{s_1}) and (\ref{s_2}).

We will investigate $s_1(0, \varphi )$. (The case of $s_2(t, \varphi )$ is
analogous.) Recall that $|H{\tilde{A}}_4| \le |H{\tilde{A}}_3|$, which implies
$h_{t,S_3}=|A_1(0, \varphi )A_4(0, \varphi )| \le
|A_1(0, \varphi )A_3(0, \varphi )| = h_{t,S_4}$. For $t$ fixed but $\varphi
\in [0, \pi ] $ variable, the length of the third side of the triangle
$\Delta A_1(0, \varphi )\allowbreak 
A_3(0, \varphi )\allowbreak A_4(0, \varphi )$ lies in the range
$$
|A_3(0, \varphi )A_4(0, \varphi )| \in [h_{t,S_4} - h_{t,S_3},
h_{t,S_4} + h_{t,S_3}] .
$$
Therefore, to show
\begin{equation}\label{s_1(0,0)}
s_1(0,\varphi ) \ge s_1(0, 0 ),
\end{equation}
we must
show the following.
Let $a:=
|A_2(0, \varphi )A_4(0, \varphi )|$ and $b:=
|A_2(0, \varphi )A_3(0, \varphi )|$. Then $t \le a \le b$,
since 
$$
\begin{cases}
\cosh t \le \cosh a =\cosh t  \allowbreak\cdot\allowbreak 
\cosh|A_1(0,\varphi )A_4(0, \varphi )|\allowbreak \\
\le \allowbreak\cosh t  \cdot 
\cosh |A_1(0,\varphi ) A_3(0, \varphi )|=\cosh b .
\end{cases}
$$
Let $
c:= 
|A_1(0, \varphi )A_3(0, \varphi )|\allowbreak-\allowbreak
|A_1(0, \varphi )A_4(0, \varphi )|=h_{t,S_4} - h_{t,S_3}$.
Then the area of the triangle with sides $a,b,c$
is less than or equal to the area of the triangle with the same first two
sides $a,b$ and with third side in the interval
$$
[h_{t,S_4}-h_{t,S_3},h_{t,S_4}+h_{t,S_3}] \subset
[h_{t,S_4}-h_{t,S_3}, 2h_{t,S_4}] \subset
[h_{t,S_4}-h_{t,S_3}, {\rm{const}}].
$$

For the last inclusion observe the following. By the geometric interpretation,
if $S_4$ is fixed and $t$ is above the bound
$ 2\arsinh[  (\tan S_4)/2  ]$ from~(\ref{tS_4}) and  increases,
then $h_{t,S_4}$ decreases. Therefore, $h_{t,S_4}$ {\emph{remains bounded for
fixed $S_4$ if $t$ increases from its originally chosen value $t_0$, say, to
infinity}}.

Inequality (\ref{s_1(0,0)}) is proved if we show the following
monotonicity property.
Fixing the
first two sides $a,b$
and varying the third side $x$ in the interval
$[h_{t,S_4}-h_{t,S_3},$ const$]$, the area is a monotonically
increasing function of $x$. 

Now we apply Lemmas \ref{BKM1} and \ref{BKM2} to the triangle 
with sides $a,b,x$.
We need to show that its area is increasing for $x \in [b-a, $ const$]$, where
we know from the preceding considerations that $0 \le b-a \le $ const. 
By these Lemmas, this area-increasing property
is satisfied for $x \in [b-a, x_{\max}]$,
where $x_{\max}$ is defined by
$\cosh ^2 (x_{\max}/2)=( \cosh a + \cosh b  ) /2$.
Thus, to complete the argument, it suffices to show that
$x_{\max} \ge $ const, i.e., that
$x_{\max} \to \infty $ for $t \to \infty $.

We estimate $x_{\max}$ from below. We have
$$
\cosh ^2 (x_{\max}/2) = \left( \cosh a  + \cosh b  \right) /2 \ge
\cosh a  > e^a / 2 ,
$$
hence
$$
(e^{x_{\max}/2})^2 > \cosh ^2 (x_{\max}/2) > e^a / 2 ,
$$
and hence
$$
x_{\max} > a - \log 2 \ge t - \log 2 \to \infty ,
$$
as we wanted to show. Thus, (\ref{eqq4}) is proved.

Taking in account inequalities (\ref{eqq1}), (\ref{eqq2}), and (\ref{eqq4}),
we can apply Lemma~\ref{index}  with $(u_1,u_2,v_1,v_2)=(0,-1,-1,0)$
to find a point $(x,\varphi)\in [0,t]\times [0,\pi]$ such that
$f_1(x,\varphi)=f_2(x,\varphi)=0$. This means that $s_1(x,\varphi)=S_1$ and
$s_2(x,\varphi)=S_2$
for the corresponding (possibly degenerate) tetrahedron $T$.

{\bf{3.}}
It remains to exclude degeneration of our tetrahedron. Our construction
yields degenerate tetrahedra only for $\varphi =0$ and $\varphi = \pi $.
In the first case, $S_4=S_1+S_2+S_3$, which contradicts our hypotheses. In the
second case, $S_4+S_3=S_2+S_1$, which implies $\pi > S_4=S_3=S_2=S_1 > 0$. 
(By the way, this can occur only for case~(\ref{fc1}) of the theorem.)
 Then a
suitable regular tetrahedron satisfies the conclusion of the theorem.
\end{proof}

%%%%%%%%%%%%%%%%%%%%%%%%%%%%%%%%%%%%%%%%%%%%%%%%%%%%%%%%%%%%%%%%%%%%%%%%%%%5

\begin{remark} Let us apply the construction in the proof of
Theorem~\ref{exhyp}
to the numbers $S_i \varepsilon ^2$ and
$t \varepsilon $ rather than $S_i$ and $t$,
where $\varepsilon \to 0$. Then, for sufficiently small $\varepsilon >0$,
hypothesis (\ref{fc1}) from Theorem~\ref{exhyp}
holds, and as an analogue of (\ref{tS_4}) we have
$2\sinh (t\varepsilon /2)> \tan (S_4 \varepsilon ^2)$.  In the
limit, we obtain a Euclidean tetrahedron with facet areas $S_i$ and
one edge
of length $t$.
Letting $t \to \infty $
gives another proof for the last statement of Theorem E
for ${\mathbb R}^3$ (existence of tetrahedra of arbitrarily small positive
volume). Namely,
the heights of the two facets meeting at the edge of length $t$,
corresponding to this edge, 
are $O(1/t)$. Thus, the tetrahedron is included
in a right circular cylinder of height $t$ and radius $O(1/t)$.
Hence, the volume of the tetrahedron
is $O(1/t)$. 
Degeneration is excluded as in
Step~{\bf{3}} of the above proof of Theorem~\ref{exhyp}.
\end{remark}

%%%%%%%%%%%%%%%%%%%%%%%%%%%%%%%%%%%%%%%%%%%%%%%%%%%%%%%%%%%%%%%%%%%%%%%%%%%%%
%%%%%%%%%%%%%%%%%%%%%%%%%%%%%%%%%%%%%%%%%%%%%%%%%%%%%%%%%%%%%%%%%%%%%%%%%%%%%
%%%%%%%%%%%%%%%%%%%%%%%%%%%%%%%%%%%%%%%%%%%%%%%%%%%%%%%%%%%%%%%%%%%%%%%%%%%%%

\section{Proofs for the spherical case}

%%%%%%%%%%%%%%%%%%%%%%%%%%%%%%%%%%%%%%%%%%%%%%%%%%%%%%%%%%%%%%%%%%%%%%%%%%%%%

Recall our convention about the notion of \emph{simplices in} ${\mathbb{S}}^n$ 
at the beginning of Section~\ref  {spherical}.

%%%%%%%%%%%%%%%%%%%%%%%%%%%%%%%%%%%%%%%%%%%%%%%%%%%%%%%%%%%%%%%%%%%%%%%%%%%%%

\subsection{Proof of Proposition \ref{sphnec}}
\label{6.1}
{\bf{1.}} We begin with the proof of the first inequality.
Let the facets of $P$ be
$F_i$. Their areas satisfy the equation
\begin{equation}\label{inin}
S_i={\text{const}}_n \cdot \int |F_i \cap \mathbb{S}^1|\, d\mathbb{S}^1.
\end{equation}
Here $| \cdot |$ denotes cardinality
and ${\text{const}}_n >0$.
The integration is taken with respect to the unique $O(n+1)$-invariant
probability Borel
measure (for the standard embedding ${\mathbb{S}}^n \subset
{\mathbb{R}}^{n+1}$) on
the manifold of all great-$\mathbb{S}^1$'s in $\mathbb{S}^n$
\cite[Ch. 18, \S6, 1]{San}. 

In the integration, we may disregard those $\mathbb{S}^1$'s
that lie
in the great-$\mathbb{S}^{n-1}$'s spanned by the facets $F_1,\ldots ,F_m$, 
since they have
measure
$0$. By the same reason, we may disregard those $\mathbb{S}^1$'s that pass
through the relative boundary of $F_i$ (in the great-$\mathbb{S}^{n-1}$
spanned by
it) for all $i \in \{ 1,\dots , m \} $ simultaneously.
If an $\mathbb{S}^1$ does not lie in the
above great ${\mathbb{S}}^{n-1}$'s
and does not intersect the above relative boundaries then it
cannot contain two opposite
points of any $F_i$. Namely, 
since $F_i$ lies in a closed half-${\mathbb{S}}^{n-1}$,
both of these points would otherwise lie in the relative boundary of $F_i$ 
taken with
respect to the great-${\mathbb{S}}^{n-1}$ spanned by it.
If such an $\mathbb{S}^1$ enters $P$ at some point
$p \in F_m$ it
must also
leave $P$, through some other facet (since this ${\mathbb{S}}^1$ does not
contain two opposite points of $F_m$)
till it comes back to $p$. This
holds even in the degenerate case, that is,
when some portion of $F_m$ is a doubly
counted boundary of $P$ either as a ``flat'' piece of $P$ or as bounded
from both sides by the interior of $P$.

These considerations imply that the integral (\ref{inin}) for $i=m$
is at most the sum of the integrals for all $1 \le i \le
m-1$. Thus, (\ref{inin}) gives our inequality.

Now assume that
$P$ lies in an open half-$\mathbb{S}^n$ (in the open northern hemisphere, say)
but does not degenerate to the
doubly-counted facet $F_m$.

Let ${\mathbb{S}}_m$ be the great-${\mathbb{S}}^{n-1}$ spanned by $F_m$. If
$\cup _{i=1} ^{n-1} F_i \not\subset {\mathbb{S}}_m$, then there exists an
$x \in \cup _{i=1} ^{m-1} {\text{rel\,int}}\, F_i$ such that
$x \not\in {\mathbb{S}}_m$. Also, there exists a $y \in {\mathbb{S}}_m
\setminus F_m$ that also lies in the open northern hemisphere. Then
$x \ne y$, and
since both
$x$ and $y$ lie in the open northern hemisphere, the
great-${\mathbb{S}}^1$  $xy$ through these two points exists.

The set of ${\mathbb{S}}^1$'s transversally intersecting $\cup _{i=1}
^{m-1}{\text{\rm{rel\,int}}}
F_i$ but not intersecting $F_m$ contains some neighbourhood of the
great-${\mathbb{S}}^1$ $xy$ in the set of all great-${\mathbb{S}}^1$'s
and has therefore  positive measure.
This implies the strict inequality in this case.

If $\cup _{i=1} ^{n-1} F_i \subset {\mathbb{S}}_m$, then in both cases
$\cup _{i=1} ^{n-1} F_i \not\subset F_m$ and
$\cup _{i=1} ^{n-1} F_i \subset F_m$, we have strict inequality unless $P$
degenerates to the doubly counted facet $F_m$. However, this degeneration 
was excluded.

{\bf{2.}} We turn to the proof of the second inequality. We again use the
formula (\ref{inin}). Again we disregard those $\mathbb{S}^1$'s that lie in
the great-$\mathbb{S}^{n-1}$'s spanned by any facet $F_i$
of $P$, as well as those
$\mathbb{S}^1$'s that pass through the relative boundary of any facet $F_i$ of
$P$. We compare the sum
of the right hand sides of (\ref{inin}) for all $1 \le i \le m$ with the
analogous integral
when in the right hand side of (\ref{inin})
we take a great-$\mathbb{S}^{n-1}$ rather than $F_i$.

Clearly, for a great-$\mathbb{S}^{n-1}$, the cardinality of its intersection
with a great-$\mathbb{S}^1$ is almost always $2$. For the $\mathbb{S}^1$'s
that were not disregarded and for any $i$,
the cardinality $|F_i \cap \mathbb{S}^1|$ is at most $1$, since a
great-${\mathbb{S}}^1$ cannot contain two opposite points of $F_i$
(see part~{\bf{1}}
of this proof).
If $P$ is not degenerate,
one great-$\mathbb{S}^1$ cannot transversally intersect the interiors of three
facets $F_i$. Namely, at each of these intersection points, it passes either
into $P$ or out of $P$
(with some definite orientation of our ${\mathbb{S}}^1$).
Thus, there would be at least
four points of intersection, and the intersection of $P$ and this
great-$\mathbb{S}^1$ would be the union of at least two disjoint
non-trivial arcs. However, this
contradicts convexity of $P$.
Hence, the sum of the integrands in 
(\ref{inin}) over all facets $i=1,\dots ,m$ is at most $2$. For
degenerate $P$, the same inequality holds.
This implies
the second inequality of the proposition.

If $P$ lies in an open half-$\mathbb{S}^n$, then the set of
$\mathbb{S}^1$'s intersecting the boundary
of the open half-$\mathbb{S}^n$ but not intersecting
$\cup _{i=1}^m F_i$
has a positive measure. (Namely, any great-$\mathbb{S}^1$ sufficiently close
to the boundary of the open half-$\mathbb{S}^n$ has this property.)
This implies the
strict inequality in this case.
\qed 

%%%%%%%%%%%%%%%%%%%%%%%%%%%%%%%%%%%%%%%%%%%%%%%%%%%%%%%%%%%%%%%%%%%%%%%%%%%%%

\begin{remark}
Part {\bf{1}} of the above proof of Proposition \ref{sphnec}
extends also for ${\mathbb{H}}^n$ and yields
Proposition \ref{necne},
however without the case of equality. We have to
use also
\cite[Ch. 18, \S 6, 1]{San}, and instead of $\mathbb{S}^1$,
we have to take a segment of
a fixed positive length $t$ and then let $t$ tend to infinity.
However, we preferred to give the elementary proof for Proposition \ref{necne}.
\end{remark}

%%%%%%%%%%%%%%%%%%%%%%%%%%%%%%%%%%%%%%%%%%%%%%%%%%%%%%%%%%%%%%%%%%%%%%%%%%%%%

\begin{remark}
\label{t1-t3}
Clearly, the argument for the inequality $t_1 \le t_3$ in the proof of
Proposition \ref{nonexhyp}
is valid also for ${\mathbb{R}}^3$ and $\mathbb{S}^3$. However, for
${\mathbb{R}}^3$, this inequality is easy to show, see below.
It is also easy to show for $\mathbb{S}^3$, provided that
each facet is contained in a closed
half-${\mathbb{S}}^2$ and has at least three sides --- in particular if the
polyhedron is contained in an open half-${\mathbb{S}}^3$ --- see below.
For ${\mathbb{R}}^3$, the sum of
the angles of the $k_i$-gon $F_i$ is $t_1=(k_i-2) \pi $. Also, $F_i$ has $k_i$
neighbouring faces,
each of which has angle  sum at least $\pi $, so $t_1=(k_i-2) \pi <
k_i \pi \le t_3$. Similarly,
for $\mathbb{S}^3$, with the above hypotheses, the sum of the angles of $F_i$
is $t_1=S_i+(k_i-2) \pi $, while every other facet $F_j$
has an angle sum $S_j+(k_j-2)\pi \ge S_j + \pi$, since $k_j \ge 3$. 
So the sum of
the angles of the other facets $F_j$ is at least $\sum _{j\ne i}
(S_j + \pi )$,
and hence,
$\sum _{j\ne i} S_j + k_i \pi \le t_3$. 
Therefore, $t_1=S_i+(k_i-2) \pi < S_i + k_i \pi 
\le (\sum _{j\ne i} S_j) + k_i \pi \le
t_3$. Here, we used the first inequality of Proposition \ref{sphnec},
which implies $S_i \le \sum _{j\ne i} S_j$, provided every facet lies
in some closed half-${\mathbb{S}}^2$.
\end{remark}

%%%%%%%%%%%%%%%%%%%%%%%%%%%%%%%%%%%%%%%%%%%%%%%%%%%%%%%%%%%%%%%%%%%%%%%%%%%%%

\subsection{Proof of Theorem \ref{exsph}}

Now we give the spherical analogues of Lemmas \ref{hs1} and~\ref{hs4}.

\begin{lemma}\label{sphs1}
  The area $S$ of a right triangle $\Delta ABC \subset \mathbb{S}^2$ 
  with angle $\angle ACB=\pi/2$ and side lengths $|AC|=b$ and
  $|BC|=a$ (where $0<a,b \le \pi $)
  fulfills the equation
$$
\tan S =\frac{\sin a \cdot \sin b }{\cos a +\cos b } 
$$
if $a+b\ne \pi $.
For $a+b= \pi $, the area is $S = \pi /2$.
\end{lemma}

%%%%%%%%%%%%%%%%%%%%%%%%%%%%%%%%%%%%%%%%%%%%%%%%%%%%%%%%%%%%%%%%%%%%%%%%%%%%%

\begin{lemma}\label{sphs4}
Let $0<d \le \pi /2$.
Assume that $\Delta ABC \subset \mathbb{S}^2$ is a triangle such that
$|AB|\leq d$ and $|AC|\leq d$.
Then the area $S$ of this triangle is bounded above by the inequality
$$
S \leq 2\arctan \frac{1-\cos d }{2\sqrt{\cos d }}
$$
if $d \ne \pi/2 $.
For $d= \pi /2$ we have $S \le \pi $.
\end{lemma}

%%%%%%%%%%%%%%%%%%%%%%%%%%%%%%%%%%%%%%%%%%%%%%%%%%%%%%%%%%%%%%%%%%%%%%%%%%%%%

\begin{proof}[Proof of Lemmas \ref{sphs1} and \ref{sphs4}]
The case $a+b= \pi $ of Lemma~\ref{sphs1}
and the case
$d= \pi /2$
of Lemma~\ref{sphs4} is elementary.
In the remaining cases $\cos d >0$ for Lemma~\ref{sphs4}, and
we proceed analogously as in Lemmas \ref{hs1} 
and \ref{hs4}.
\end{proof}

%%%%%%%%%%%%%%%%%%%%%%%%%%%%%%%%%%%%%%%%%%%%%%%%%%%%%%%%%%%%%%%%%%%%%%%%%%%%%

To prove Theorem~\ref{exsph}, we use an analogous construction as for
Theorem~\ref{exhyp}.

{\bf{Construction 2.}}
Let 
\begin{equation}\label{S2}
S \in (0, \pi /2 ],
\end{equation} 
and choose
$$
t=\pi /2.
$$
This choice is motivated as follows.
We will 
apply
Lemma~\ref{BKM1} to triangles with sides $a,b \le \pi /2$,
but Lemma~\ref{BKM1} does not hold for $a=b \in ( t, \pi ) = ( \pi /2, \pi )$. 
Thus,
$t= \pi /2$ is the largest value
for which 
our proof applies. 
Also, $S$ will be the area of a spherical
triangle contained in a spherical triangle with three sides $\pi /2$, which
explains the constraint (\ref{S2}).
Now, we define a function
\begin{equation}\nonumber 
f_{t,S}\colon [0,t]\rightarrow \mathbb{R}
\end{equation}
as follows. For any $x\in [0,t]$,
consider the function
$$
g_x(y):=\arctan \frac{\sin x \cdot \sin y }{\cos x +\cos y }+
\arctan \frac{\sin(t-x)\cdot \sin y }{\cos(t-x)+\cos y } ,
$$
defined for $y \in [0, \pi /2]$.
It is easy to see that $(d/dy)g_x(y)>0$
for $y\in [0,\pi /2 )$,
$g_x(0)=0$, and
$$
g_x( \pi /2)=\pi /2 \ge S ,
$$
for all $x \in [0,t]$.
Therefore, there exists a unique $\tilde{y}\in (0,\pi /2]$
such that $g_x(\tilde{y})=S$. We put
$$
f_{t,S}(x):=\tilde{y}\in (0,\pi /2].
$$
Now we investigate some properties of this function.
Obviously $f_{t,S}$ is continuous on $[0,t]$ (moreover, is $C^1$  on
$(0,t)$), and
$f_{t,S}(x)=f_{t,S}(t-x)$.

Here is the geometric interpretation of $f_{t,S}$. Consider a 
triangle $\Delta ABC \subset \mathbb{S}^2$ with the
following properties.
\begin{enumerate}
\item  $|AB|=t$,
\item the area of $\Delta ABC$ is $S$,
\item
$C$ has an orthogonal projection $H$
to the line
$\ell (A,B)$ such that $H$ lies in the segment $[A,B]$ and
$|AH|=x\in [0,t]$. (Observe that there are at least two orthogonal
projections of $C$ to $\ell (A,B)$.)
\end{enumerate}
Then it is easy to see
(using Lemma~\ref{sphs1}) that 
 $|CH|=f_{t,S}(x)$.
It is also easy to see that for $0 < \tilde{S}<S$
and for every $x\in [0,t]$,
 we have
$f_{t,\tilde{S}}(x)<f_{t,S}(x)$.

The boundary values
$$
h_{t,S}:=f_{t,S}(0)=f_{t,S}(t)
$$ 
are easy to determine.
By the geometric interpretation, this is the
third side of a spherical triangle with two other
sides of length $\pi /2$ and area $S$, i.e.,
$$
h_{t,S}=S.
$$

%%%%%%%%%%%%%%%%%%%%%%%%%%%%%%%%%%%%%%%%%%%%%%%%%%%%%%%%%%%%%%%%%%%%%%%%%%%%%

\begin{proof}[Proof of Theorem~\ref{exsph}]
{\bf{1.}}
First we consider the case when hypothesis (\ref{fc1'}) of Theorem~\ref{exsph} 
holds.
We roughly follow the lines of the proof of Theorem~\ref{exhyp} for
the analogous case when
hypothesis~(\ref{fc1}) holds.

We have
\begin{equation}\label{tan}
\frac{1-\cos h_{t,S_4}}{2 \sqrt{\cos h_{t,S_4}}}=
\frac{1-\cos S_4}{2 \sqrt{\cos S_4}} \le \tan(S_1/2),
\end{equation}
by the hypothesis of the theorem.
From this point onwards, the construction is the same as in 
Theorem~\ref{exhyp}.

We still have to show that our tetrahedron $T$ satisfies our convention 
about simplices in ${\mathbb{S}}^n$ (see the beginning 
of Section~\ref  {spherical}).
A convex combinatorial simplex in an \emph{open}
half-${\mathbb{S}}^n$ is always considered as a simplex in ${\mathbb{S}}^n$.
Let $\mathbf{e}_1,\mathbf{e}_2,\mathbf{e}_3,\mathbf{e}_4$ denote the usual 
unit basis
vectors,
and let $x_1,x_2,x_3,x_4$ denote the corresponding coordinates. 
We set $A_1:={\bf{e}}_1$ and $A_2:={\bf{e}}_2$.
The
rotation about $\ell (A_1,A_2)$ in ${\mathbb{S}}^3$
maps $A_3:={\mathbf{e}}_3$ to ${\mathbf{e}}_3 \cos \varphi +
{\mathbf{e}}_4 \sin \varphi $, 
say. Then $T$ is in
the closed half-${\mathbb{S}}^3$ defined by the inequality $x_1+x_2 \ge 0 $.
Moreover, 
if $S_3 \le S_4
< \pi /2$, then $T$ is contained in the open half-${\mathbb{S}}^3$ 
defined by the inequality $x_1+x_2>0$, and we are done. 
If $S_3 \le  S_4 =
\pi /2$, a slight perturbation of the open half-${\mathbb{S}}^3$
given by $x_1+x_2>0$
contains $T$ for all $\varphi \in [0, \pi ]$, and we are done.
If $S_3=S_4= \pi /2$,
and $0
\le \varphi < \pi /2$ is {\emph{fixed}}, 
then also a slight perturbation of the open    
half-${\mathbb{S}}^3$ given by $x_1+x_2>0$ contains $T$ and then 
we are also done.
The case
$S_3=S_4=\varphi = \pi /2$ will be treated in part~{\bf{3}} below. 

We have to observe that from the construction, we have
$|A_3(x, \varphi )
A_4(x, \varphi )| \le f_{t,S_3}(x)+f_{t,S_4}(x) \le 2f_{t,S_4}(x) \le
\pi $. Thus, the edge $[A_3(x, \varphi ),
A_4(x, \varphi )]$ of our tetrahedron is in the closed
angular domain swept by $\gamma ( \varphi )\sigma ^+$ for $\varphi \in [0,
\pi ]$, as in the hyperbolic case. 
(This explains the inequality
$S_4 \le \pi /2$ of the theorem --- and thus
also the inequality $S \le \pi /2$ in the construction:
without this inequality of the theorem, the last sentence would not be valid.
Moreover, let $S_3=S_4 \in ( \pi /2, \pi )$. Then consider
$|A_3(x, \varphi ) A_4(x, \varphi )| $ defined not as a distance
but defined by analytic continuation from $\varphi $'s close to $0$, i.e., by
retaining the geometry of the figure.
Then 
for $\varphi = \pi $
we would have
$|A_3(x, \varphi ) A_4(x, \varphi )| = 2f_{t,S_4}(x)  \in ( \pi , 2 \pi )$. This
would imply that the tetrahedron $T$ defined by the same analytic 
continuation, i.e., by
retaining the geometry of the figure, would not be convex.)

We define $s_i(x, \varphi )$ (for $1 \le i \le 4$) and $f_i(x, \varphi )$
(for $i=1,2$) as in the proof of Theorem~\ref{exhyp}.
The formulas
\begin{equation}\label{eqq1'}
f_1(x,0)+f_2(x,0)=S_4-S_1-S_2-S_3 < 0,
\end{equation}
\begin{equation}\label{eqq2'}
f_1(x,\pi)+f_2(x,\pi)=S_3+S_4-S_1-S_2 \geq 0
\end{equation}
follow like (\ref{eqq1}) and (\ref{eqq2})
in the hyperbolic case.
The formulas
\begin{equation}\nonumber
f_1(0,\varphi) \le 0, \quad f_2(t,\varphi) \le 0,
\end{equation}
for all $\varphi \in [0,\pi]$
follow from Lemma~\ref{sphs4} similarly as in the hyperbolic case from
Lemma~\ref{hs4}. (Observe that here we have non-strict inequalities. Namely, 
in hypothesis~(\ref{fc1'})
of Theorem~\ref{exsph} and in (\ref{tan}), we have non-strict
inequalities,
whereas for the
hyperbolic case,
 we had
strict inequalities
 in hypothesis (\ref{fc1}) of Theorem~\ref{exhyp} and in~(\ref{tdef}).)
Then, 
 we choose
$(u_1,u_2,v_1,v_2):=(1,0,0,1)$ and finish the proof of case~{\bf{1}}
as in the hyperbolic case. Also here, the tetrahedron $T$
is possibly degenerate. 
(Observe that allowing $S_4 > \pi /2$, we could have $h_{t,S_4} > \pi /2$.
Then $|A_1A_3(0, \varphi )|, |A_1A_4(0, \varphi )| \le h_{t,S_4}$ makes
 it impossible to apply Lemma~\ref{sphs4}. 
This explains once more the inequality
$S_4 \le \pi /2$ of the theorem --- and thus
also the inequality $S \le \pi /2$ in the construction --- for case {\bf{1}}.)

{\bf{2.}}
Now we consider the case when hypothesis 
(\ref{fc2'}) of Theorem~\ref{exsph} holds.
We roughly follow the lines of the proof of Theorem~\ref{exhyp} 
when hypothesis~(\ref{fc2}) holds. 

As in Step~{\bf{1}} of this proof, our tetrahedron satisfies
our convention about the notion of a simplex in ${\mathbb{S}}^n$
(defined the beginning of Section~\ref  {spherical}),
unless $S_3=S_4=\varphi = \pi /2$.  This last case
will be handled below in part~{\bf{3}} of this proof.

We obtain the inequalities (\ref{eqq1'}) and (\ref{eqq2'})
similarly to (\ref{eqq1}) and (\ref{eqq2}) in the hyperbolic case.

Now we check that
\begin{equation*}\nonumber
f_2(0,\varphi)\geq 0, \quad f_1(t,\varphi) \geq 0,
\end{equation*}
for all $\varphi \in [0, \pi ]$. As in the hyperbolic case, this reduces to
showing that
\begin{equation}\label{eqq5'}
s_1(0,\varphi)\geq s_1(0,0), \quad s_2(t,\varphi ) \geq s_2(t,0).
\end{equation}
We will investigate $s_1(0, \varphi )$. (The case of $s_2(t, \varphi )$ is
analogous.)
Observe that the distance $|A_3(x, \varphi )A_4(x, \varphi )|$ is a strictly 
increasing function of $\varphi \in [0, \pi ]$, with
$$
|A_3(x, 0 )A_4(x, 0 )|=h_{t,S_4}-h_{t,S_3}=S_4-S_3 ,
$$
and
$$
|A_3(x, \pi )A_4(x, \pi )|=h_{t,S_4}+h_{t,S_3}=S_4+S_3 .
$$
By the geometric interpretation (observe that 
$|A_2(0, \varphi )A_1(0, \varphi )|=
|A_2(0, \varphi )A_3(0, \varphi )|=|A_2(0, \varphi )A_4(0, \varphi )|= \pi /2$),
we have
$$
s_1(0, \varphi )=
\angle A_3(0, \varphi )A_2(0, \varphi ) A_4(0, \varphi )=
|A_3(0, \varphi )A_4(0, \varphi )|,
$$
where the last term is strictly increasing for $\varphi \in [0, \pi ]$. This
shows (\ref{eqq5'}). 
Then, as in the hyperbolic case, we choose
$(u_1,u_2,v_1,v_2):=(0,-1,-1,0)$ and finish the proof of~{\bf{2}}. 
Also here, the tetrahedron $T$ can be degenerate.

{\bf{3.}}
It remains 
\newline
1) to exclude degeneration of our tetrahedron, and 
\newline
2) to verify that
our tetrahedron satisfies
our convention about the notion of a simplex in ${\mathbb{S}}^n$
(see the beginning of Section~\ref {spherical}).

1) is done exactly as in the hyperbolic case in Theorem~\ref{exhyp}. Here
we can even have $\pi \ge S_4=S_3=S_2=S_1>0$. 

For 2) we have to handle the case $S_3 = S_4 = \varphi = \pi /2
$ only. 
From Construction 2 for any $x \in [0,t] = [0, \pi /2 ]$ we have $S \le \pi
/2$, where equality can be attained for any $x \in [0,t]$. The case of
equality is independent of $x \in [0,t]$: namely it is a regular
spherical triangle with angles and sides $\pi /2$. Then the fact that the
angle of the facets $A_1(x, \varphi )A_2(x, \varphi )A_3(x, \varphi )$
and $A_1(x, \varphi )A_2(x, \varphi )A_4(x, \varphi )$ is $\varphi = \pi /2$
uniquely determines our simplex: it is a regular simplex in
${\mathbb{S}}^3$ of edge $\pi /2$ (thus we have also $S_1=S_2= \pi /2$ --- the
vertices can be ${\mathbf{e}}_1, \ldots , {\mathbf{e}}_4$). 
This lies in some open
half-${\mathbb{S}}^3$, therefore is among the simplices that we considered as
simplices in ${\mathbb{S}}^3$.
\end{proof}

%%%%%%%%%%%%%%%%%%%%%%%%%%%%%%%%%%%%%%%%%%%%%%%%%%%%%%%%%%%%%%%%%%%%%%%%%%%%%

\subsection{Proof of Proposition \ref{sphsuff}}
For part~(i),
we start with $n=2$ dimensions. Here, one has a convex $m$-gon in a 
closed half-${\mathbb{S}}^2$,
with sides $S_1, \dots ,S_m$. In fact, it lies in an open half-${\mathbb{S}}^2$,
has strictly convex angles, and is non-degenerate if
$S_m < S_1+ \dots +S_{m-1}$ and $S_1+ \dots +S_m < 2\pi $. For the
degenerate cases, i.e., when $S_m=S_1+ \dots +S_{m-1}$ or $S_1+ \dots +S_m =
2 \pi $, we have a doubly counted segment or a great-${\mathbb{S}}^1$,
respectively. If both equations hold, then we have also a digon, with one
side subdivided to $m-1$ sides. If both
inequalities are strict, we can copy the well-known
proof in \cite[pp.~53--54]{Kr}  --- given there
for the case of ${\mathbb{R}}^2$. Thus, we obtain the existence of such a
convex $m$-gon. Actually, one gets such a convex $m$-gon that is
inscribed in a circle of radius less than $ \pi /2 $.

Now we show how this construction can be lifted to higher dimensions.
For ${\mathbb{S}}^3$, we embed 
the above $m$-gon in its equator, which is an ${\mathbb{S}}^2$.
Each side of
this polygon is then replaced by a facet that is
the union of all meridians (whose lengths are
$\pi $)
meeting that side. The vertices are replaced similarly by edges that are
meridians meeting these vertices.
Additionally, there are two new vertices at the North and South Poles. Then the
ratio of the areas of the spherical digons and the lengths of the
corresponding edges of our polygon
is $V_2({\mathbb{S}}^2)/V_1({\mathbb{S}}^1)$, where $V_i$
denotes $i$-volume. Moreover, the dihedral angles are the same as for the
spherical $m$-gon in ${\mathbb{S}}^2$.

The inductive step is performed analogously for all $n >3$.
The other stated properties are obvious.

 Part~(ii) about simplices is proved by induction on $n$.
For $n=2$, we have a spherical triangle, and we have the same 
degenerate cases as
in part~(i) for $m=n+1=3$.
 Let $n \ge
3$, and assume that the statement of the theorem holds for $n-1$.
With the factor $\alpha := 
V_{n-2}({\mathbb{S}}^{n-2})/V_{n-1}({\mathbb{S}}^{n-1})$,
the numbers
$\alpha(S_1+S_2), \alpha S_3, \ldots,\alpha S_{n+1}$
satisfy the
hypotheses of the proposition for $n-1$.
Arguing as in the second proof of 
Theorem~\ref{main2} in Section~\ref{sec:second} part {\bf{1}},
we establish that
 $S_{\text{new}}:=S_1+S_2 \le S_3+ \dots +S_{n+1}$,
since $n+1 \ge 4$, and, for $j \ge 3$, $S_j \le S_{\text{new}} + S_3 + \ldots
+ S_{j-1}+S_{j+1}+ \ldots + S_{n+1}$.

Therefore, we have on ${\mathbb{S}}^{n-1}$ a polyhedral complex that is a
combinatorial simplex with these facet areas. Again we consider
${\mathbb{S}}^{n-1}$ as the equator of ${\mathbb{S}}^n$.
We replace
each facet and each lower-dimensional face of
this polyhedral complex on ${\mathbb{S}}^{n-1}$
 by the union of all meridians
(whose lengths are $\pi $) meeting it. 
The resulting facets and also lower-dimensional faces are of one dimension 
higher than the original ones.
Additionally, there are two new vertices at the North and South Poles.
Thus, we have obtained a polyhedral complex on ${\mathbb{S}}^n$ with facet areas
$S_1+S_2, S_3, \dots , S_{n+1}$.

This polyhedral complex has
only two vertices at the two poles, and
 $n$ edges joining them. 
Its $n$ facets are obviously not simplices.
The facet of area $S_1+S_2$ has $n-1$ edges.
On each of these $n-1$ edges, we add an extra vertex at the same geographic
latitude. Also we add an extra (convex)
simplicial $(n-2)$-face
with these vertices, together with its faces of lower dimensions, 
that subdivides the facet of area $S_1+S_2$.
For a suitable choice of the
latitude, the facet of area $S_1+S_2$ is subdivided into two new
$(n-1)$-dimensional
simplicial facets of areas
$S_1$ and $S_2$. Omitting the facet of area $S_1+S_2$, with all its faces of
positive dimension, 
these two new facets with all their lower-dimensional faces are added as well.
In each of the other facets (of areas $S_3, \dots ,S_{n+1}$),
 one $(n-2)$-face has been subdivided into two new simplicial
 $(n-2)$-faces.
Thus, these other facets also become
combinatorial $(n-1)$-simplices, by induction with respect to $n$.

The other stated properties follow by the construction.
\qed 

{\bf{Added in proof}}, 20. Oct. 2014. 
Observe that (i) of Theorem F$'$ is preserved when passing to weak$^*$ limits,
so this will present no problem. However, (ii) of Theorem F$'$ does not hold
automatically, we have to prove it for $K'$ (and for $K''$, but that follows
from the considerations for $K'$). For $K$ we have (ii) of Theorem F$'$, that
can be rewritten as $\int _{S^{n-1}} | \langle u, u_0 \rangle | d \mu _K >0$,
for each $u_0 \in S^{n-1}$. Since this integral is a continuous (actually
Lipschitz) function of $u_0$, by compactness 
this integral even has a positive lower bound $a$ independent of $u_0$.
Now let $\varepsilon >0$ be sufficiently small. Let $\{ u_{01}, \dots , u_{0m}
\} \subset S^{n-1}$ be an $\varepsilon $-net. Then for any $u_{0k}$ we have 
$\int _{S^{n-1}} | \langle u, u_{0k} \rangle | d \mu _K \ge a >0$. We may
suppose that each $K_{i_j}$ satisfies 
$\int _{S^{n-1}} | \langle u, u_{0k} \rangle | d \mu _ {K_{i_j}} > a/2 >0$.
For any $u_0 \in S^{n-1}$ we have $\| u_0 -u_{0k} \| < \varepsilon $ for
some $1 \le k \le m$. Then 
$| \int _{S^{n-1}} | \langle u, u_0 \rangle | d \mu _ {K_{i_j}} -
\int _{S^{n-1}} | \langle u, u_{0k} \rangle | d \mu _ {K_{i_j}} | \le
\int _{S^{n-1}} | \langle u, u_0 - u_{0k} \rangle | d \mu _ {K_{i_j}} <
\varepsilon \int _{S^{n-1}} 1 d \mu _ {K_{i_j}} $. We may also suppose that
$\int _{S^{n-1}} 1 d \mu _ {K_{i_j}} < 2 \int _{S^{n-1}} 1 d \mu _K =:2b$.
Then $ \int _{S^{n-1}} | \langle u, u_0 \rangle | d \mu _ {K_{i_j}} > 
\int _{S^{n-1}} | \langle u, u_{0k} \rangle | d \mu _ {K_{i_j}} - 2b \varepsilon
> a/2 - 2b \varepsilon >0$. Hence passing to the limit $K'$ we have 
$ \int _{S^{n-1}} | \langle u, u_0 \rangle | d \mu _ {K'} \ge 
a/2 - 2b \varepsilon >0$. Thus the measure $\mu _ {K'}$ satisfies both (i) and
(ii) of Theorem F$'$, 
hence it is the surface area measure of a convex body $K'$. 

\medskip
{\bf Acknowledgements.}
The authors are indebted to R. Schneider for pointing out in his book
\cite{Schn} the proof of Theorem G, 
used in the second proof of our Theorem~\ref{main2};
to F. Morgan for pointing out the references {\cite{Br}} and {\cite{Mo}};
and to G.~Panina for pointing out the results about the orthocentric simplices
\cite{Ge,EHM}.
The second named author (E. M., Jr.)
expresses his thanks to the Fields Institute,
where part of this work was done during a Conference on Discrete Geometry, in
September 2011.
Research of the fifth author (G.~R.)\ was
initiated at the reunion conference of the special
semester on discrete and computational geometry at the Bernoulli Center, EPFL
Lausanne, February~27--March~2, 2012.

%%%%%%%%%%%%%%%%%%%%%%%%%%%%%%%%%%%%%%%%%%%%%%%%%%%%%%%%%%%%%%%%%%%%%%%%%%%%%
%%%%%%%%%%%%%%%%%%%%%%%%%%%%%%%%%%%%%%%%%%%%%%%%%%%%%%%%%%%%%%%%%%%%%%%%%%%%%
%%%%%%%%%%%%%%%%%%%%%%%%%%%%%%%%%%%%%%%%%%%%%%%%%%%%%%%%%%%%%%%%%%%%%%%%%%%%%

\vspace{10mm}

\end{document}